\documentclass[a4paper, 10pt, twoside, notitlepage]{amsart}

\usepackage{amsmath,amscd}
\usepackage{amssymb}
\usepackage{amsthm}
\usepackage{comment}
\usepackage{graphicx, xcolor}

\usepackage{mathrsfs}
\usepackage[ocgcolorlinks, linkcolor=blue]{hyperref}

\usepackage[ocgcolorlinks,linkcolor=blue]{hyperref}

\usepackage{bm}
\usepackage{bbm}
\usepackage{url}

\newcommand{\LC}{\left(}
\newcommand{\RC}{\right)}

\theoremstyle{plain}
\newtheorem{thm}{Theorem}[section]
\newtheorem{prop}{Proposition}[section]
\newtheorem{lem}[prop]{Lemma}
\newtheorem{cor}[prop]{Corollary}

\newtheorem{defi}[prop]{Definition}
\newtheorem{rmk}[prop]{Remark}

\numberwithin{equation}{section}
\newcommand {\R} {\mathbb{R}} 
 \newcommand {\N} {\mathbb{N}}
 
\newcommand {\p} {\partial}

\newcommand{\eps}{\epsilon}

\newcommand{\wt}{\widetilde}

\newcommand{\norm}[1]{\lVert #1 \rVert}         

\title[Calder\'on problem for the fractional wave equation]{The Calder\'on problem for the fractional wave equation: Uniqueness and optimal stability}

\author[P.-Z. Kow]{Pu-Zhao Kow}
\address{Department of Mathematics, National Taiwan University, Taipei 106, Taiwan}
\email{d07221005@ntu.edu.tw}

\author[Y.-H. Lin]{Yi-Hsuan Lin}
\address{Department of Applied Mathematics, National Yang Ming Chiao Tung University, Hsinchu 30050, Taiwan}
\email{yihsuanlin3@gmail.com}

\author[J.-N. Wang]{Jenn-Nan Wang}
\address{Institute of Applied Mathematical Sciences, National Taiwan University, Taipei 106, Taiwan}
\email{jnwang@math.ntu.edu.tw}

\begin{document}
	
	\maketitle
	
	\begin{abstract}
	We study an inverse problem for the fractional wave equation with a potential by the measurement taking on  arbitrary subsets of the exterior in the space-time domain. We are interested in the issues of uniqueness and stability estimate in the determination of the potential by the exterior Dirichlet-to-Neumann map. The main tools are the qualitative and quantitative unique continuation properties for the fractional Laplacian. For the stability, we also prove that the log type stability estimate is optimal. The log type estimate shows the striking difference between the inverse problems for the fractional and classical wave equations in the stability issue. The results hold for any spatial dimension $n\in \N$.
		
		\medskip
		
		
		
	\end{abstract}

	\tableofcontents

	\section{Introduction}\label{Sec 1}
	In this paper, we study an inverse problem for the fractional wave equation with a potential. The mathematical model for the fractional wave equation is formulated as follows.
	Let $\Omega\subset \R^n$ be a bounded Lipschitz domain, for $n\in \N$. Given $T>0$,  $s\in (0,1)$ and $q=q(x)\in L^\infty(\Omega)$, consider the initial exterior value problem for the wave equation with the fractional Laplacian,
	\begin{align}\label{fractional wave main}
		\begin{cases}
			\LC \p_t^2 + (-\Delta)^s  + q \RC u =0 & \text{ in }\Omega_T :=\Omega \times (0,T),\\
			u=f  & \text{ in }(\Omega_e)_T:= \Omega_e \times (0,T),\\
			u=\p _t u=0 & \text{ in }\R^n \times \{0\},
		\end{cases}
	\end{align}
	where $(-\Delta)^s$ is the standard fractional Laplacian\footnote{A rigorous definition is given in Section \ref{sec:forward}.}, and 
	$$\Omega_e := \R^n\setminus \overline{\Omega}$$ 
	denotes the exterior domain. The fractional wave equation can be regarded as a special case of the peridynamics which models the nonlocal elasticity theory, see e.g. \cite{silling2016introduction}. 
	
	In recent years, inverse problems involving the fractional Laplacian have received a lot of attention. Ghosh-Salo-Uhlmann \cite{ghosh2016calder} first proposed the Calder\'on problem for the fractional Schr\"odinger equation, and the proof relies on the \emph{strong uniqueness} of the fractional Laplacian (\cite[Theorem 1.2]{ghosh2016calder}) and the \emph{Runge approximation} (\cite[Theorem 1.3]{ghosh2016calder}). Based on these two useful tools, there are many related works appeared in past few years, such as \cite{BGU18,CLL2017simultaneously,CL2019determining,cekic2020calderon,ghosh2017calder,GRSU18,harrach2017nonlocal-monotonicity,harrach2020monotonicity,LL2020inverse,lai2019global,RS20Calderon,LLR2019calder,lin2020monotonicity} and the references therein.

	Throughout this work, we assume that the (lateral) exterior data $f$ is compactly supported in the set $W_T := W \times (0,T)\subset (\Omega_e)_T$, 
	 where $W\subset \Omega_e$ with $\overline{W}\cap \overline{\Omega}=\emptyset$ can be any nonempty open subset with Lipschitz boundary, and, to simplify our notations, we assume that both $q$ and $f$ are real-valued functions.
	Note that the initial boundary value problem \eqref{fractional wave main} is a mixed \emph{local-nonlocal} type equation.
	In order to study the inverse problem of \eqref{fractional wave main}, we will use the strong approximation property of \eqref{fractional wave main}, which is due to  the \emph{nonlocality} of the fractional Laplacian $(-\Delta)^s$, for $0<s<1$.
	Hence, by the well-posedness of \eqref{fractional wave main} (see Theorem \ref{thm:well-posedness}), one can formally define the associated \emph{Dirichlet-to-Neumann} (DN) map $\Lambda_q$   
    \begin{equation}
   \Lambda_q : C_{c}^{\infty}((\Omega_e)_{T})\to L^2(0,T; H^{-s}(\Omega_e)), \quad  \Lambda_q : f \mapsto \left.(-\Delta)^s u\right|_{(\Omega_e)_T},\label{eq:hyperbolic-DN}
    \end{equation}
	where $u$ is the unique solution to \eqref{fractional wave main}. The precise definitions of the Sobolev spaces will be given in Section~\ref{subsec:Sobolev}.
	Let us state the first main result of our work.
	
	\begin{thm}[Global uniqueness]\label{thm:global uniqueness}
		Consider $T>0$, $s\in (0,1)$, and $q_j=q_j(x)\in L^\infty(\Omega)$, for $j=1,2$. 
		Assume that $W_1,W_2\subset \Omega_e$ are arbitrary open sets with Lipschitz boundary such that $\overline{W_1}\cap \overline{\Omega}=\overline{W_2}\cap \overline{\Omega}=\emptyset$.
		Let $\Lambda_{q_j}$ be the DN map of 
		\begin{align}\label{fractional wave in uniqueness}
			\begin{cases}
				\LC \p_t^2 + (-\Delta)^s +q_j \RC u =0 & \text{ in }\Omega_T,\\
				u=f  & \text{ in }(\Omega_e)_T,\\
				u(x,0)=\p _t u(x,0)=0 & \text{ in }\R^n \times \{0\},
			\end{cases}
		\end{align}
	for $j=1,2$. If 
    \begin{align}
    		\left. \Lambda_{q_1}(f) \right|_{(W_2)_T}=\left. \Lambda_{q_2}(f) \right|_{(W_2)_T},  \text{ for any }f\in C_{c}^\infty((W_1)_T), \label{eq:assump}
    \end{align} 
    then $q_1=q_2$ in $\Omega_T$\footnote{Throughout this paper, we adapt the notation $A_T:=A\times (0,T)$, for any set $A\subset \R^n$.}.
	\end{thm}
	
	The proof of Theorem \ref{thm:global uniqueness} is based on the qualitative form of  the \emph{Runge approximation} for the fractional wave equation: For any $g\in L^2(\Omega_T)$, there exists a sequence of functions $\{f_k\}_{k\in \N}\in C_{c}^\infty((W_1)_T)$ such that $u_{k}\to g$ in $L^2 (\Omega_T)$ as $k\to \infty$, where $u_k$ is the solution to \eqref{fractional wave main} with $u_k=f_k$ in $(\Omega_e)_T$, for all $k\in \N$. The preceding characterization can be regarded as an \emph{exterior control} approach, in the sense that one can always control the solution by choosing appropriate exterior data.
	
	The second main result of the paper is a quantitative version of Theorem~\ref{thm:global uniqueness}, which provides a stability estimate for our fractional Calder\'on problem.
	Before we state the stability result, we introduce some notations. 
	\begin{defi}
		Let $H^{2}(0,T;\wt H^{\alpha}(\Omega))$ be the Sobolev space
		equipped with the norm 
		\[
		\|u\|_{H^{2}(0,T;\wt H^{\alpha}(\Omega))}=\|u\|_{L^{2}(0,T; \wt H^{\alpha}(\Omega))}+\|\partial_{t} u\|_{L^{2}(0,T;\wt H^{\alpha}(\Omega))}+\|\partial_{t}^{2}u\|_{L^{2}(0,T;\wt H^{\alpha}(\Omega))}.
		\]
		We also denote 
		\[
		H_{0}^{2}(0,T;\wt H^{\alpha}(\Omega)):=\left\{ u\in H^{2}(0,T;\wt{H}^{\alpha}(\Omega)):\  u=\partial_{t}u=0\text{ in }\mathbb{R}^{n}\times\{0\}\right\},
		\]
		and let $H^{-2}(0,T;H^{-\alpha}(\Omega))$ be the dual space of
		$H_{0}^{2}(0,T;\wt{H}^{\alpha}(\Omega))$. We shall explain the space $\wt{H}^{\alpha}(\Omega)$ in more detail later in Section~{\rm \ref{sec:forward}}. 
	\end{defi}
	
	\begin{defi}\label{defi: space Z}
		For each $\alpha>0$ and $T>0$, we define 
		\[
	    	\|q\|_{Z^{-\alpha}(\Omega;T)}
	    	  :=\sup\begin{Bmatrix}\left|\int_{\Omega_{T}}q(x,t)\phi_{1}(x,t)\phi_{2}(x,t)\,dx\,dt\right|\end{Bmatrix},
		\]
		where the supremum is taken over all functions $\phi_{1},\phi_{2}\in C_{c}^{\infty}(\Omega_{T})$ with
		\[
		\|\phi_{j}\|_{H^{2}(0,T;\wt{H}^{\alpha}(\Omega))}=1 \quad (j=1,2),
		\]
		and let $Z^{-\alpha}(\Omega;T)$ be the Banach space equipped with this
		norm. 
	\end{defi}
	
	\begin{rmk}
		Since $\|\phi_{j}\|_{L^{2}(\Omega_{T})}\le\|\phi_{j}\|_{H^{2}(0,T;\wt{H}^{\alpha}(\Omega))}=1$
		for all $\phi \in H_0^{2}(0,T;\wt{H}^{\alpha}(\Omega))$, $\alpha>0$ and $T>0$, it is easy to see that 
		\[
		\|q\|_{Z^{-\alpha}(\Omega;T)}\le\|q\|_{L^{\infty}(\Omega_{T})}
		\]
		for all $q=q(x,t)$, which implies $L^{\infty}(\Omega_{T})\subset Z^{-\alpha}(\Omega;T)$. 
	\end{rmk}
	
	To shorten our notations, we denote the operator norm as 
	\[
	\norm{\cdot}_{*}=\norm{\cdot}_{L^{2}(0,T;H_{\overline{W}}^{2s})\rightarrow L^{2}(0,T;H^{-2s}(W))},
	\]
	where the Sobolev space $H_{\overline{W}}^{2s}$ will be described in Section~\ref{subsec:Sobolev}. We are now ready to state the second main result of our work. 
	\begin{thm}
		[Logarithmic stability] \label{thm:main-stability}Let $T>0$,
		$s\in(0,1)$, and $q_{j}=q_{j}(x)\in L^{\infty}(\Omega)$, for $j=1,2$.
		Assume that $W_{1},W_{2}\subset\Omega_{e}$ be arbitrary open sets
		with Lipschitz boundary such that $\overline{W_{1}}\cap\overline{\Omega}=\overline{W_{2}}\cap\overline{\Omega}=\emptyset$.
		Let $\Lambda_{q_{j}}$ be the DN map of \eqref{fractional wave in uniqueness}
		for $j=1,2$. We also fix a regularizing parameter $\gamma>0$. If
		$q_{1}$ and $q_{2}$ both satisfy the apriori bound 
		\[
		\|q_{j}\|_{L^{\infty}(\Omega)}\le M\quad\text{ for }j=1,2,
		\]
		then 
		\[
		\|q_{1}-q_{2}\|_{Z^{-s-\gamma}(\Omega;T)}\le\omega\LC\|\Lambda_{q_{1}}-\Lambda_{q_{2}}\|_{*}\RC
		\]
		where $\omega$  satisfies 
		\[
		\omega(t) \leq C|\log t|^{-\sigma}, \quad 0\leq t\leq 1,
		\]
		for some constants $C$ and $\sigma$ depending only on $n,s,\Omega,W_{1},W_{2},\gamma,T,M$. 
	\end{thm}
	
	Inspired by Theorem~\ref{thm:global uniqueness}, we will prove Theorem~\ref{thm:main-stability}
	by using a quantitative version of Runge approximation, which involves the well-known \emph{Caffarelli-Silvestre extension} for the fractional Laplacian and the \emph{propagation of smallness}. Moreover, Theorem \ref{thm:global uniqueness} and Theorem \ref{thm:main-stability} are satisfied for any spatial dimension $n\in \N$.
	
	The third main result of this work studies the \emph{exponential instability} of the Calder\'{o}n
	problem for the fractional wave equation. In other words, the stability
	result in Theorem~\ref{thm:main-stability} is optimal.
	For brevity, we denote the operator norm 
	\[
	\|\mathcal{A}\|_{*}'=\sup_{0\not\equiv\chi\in C_{c}^{\infty}((0,T))}\frac{\sup_{t\in(0,T)}\left\|\chi\mathcal{A}\chi\right\|_{L^{2}(B_{3}\setminus\overline{B_{2}})\rightarrow L^{2}(B_{3}\setminus\overline{B_{2}})}(t)}{\|\chi\|^2_{W^{2,\infty}(0,T)}},
	\]
	where $B_r$ with $r>0$ stands for the ball of radius $r$ centered at the origin. 
	\begin{thm}[Exponential instability I]\label{thm:main-instability} 
		Let $\Omega=B_{1}\subset \R^n$, for $n\geq 2$, $n\in \N$.
		Given any $T>0$, $s\in(0,1)$, $\alpha>0$ and $R>0$. There exists
		a positive constant $c_{R,T,n,s}$ such that: Given any $0<\epsilon<c_{R,T,n,s}$,
		there exist potentials $q_{1},q_{2}\in C^{\alpha}(\Omega)$ such
		that $\|q_{j}\|_{L^{\infty}(\Omega)}\le R$, $j=1,2$, satisfying
		\begin{equation}
			\|q_{1}-q_{2}\|_{L^{\infty}(\Omega)}\ge\epsilon,\label{eq:opt-discrete}
		\end{equation}
		but 
		\begin{equation}\label{eq:opt-DN-map-small}
			\left\|\Lambda_{q_{1}}-\Lambda_{q_{2}}\right\|_{\ast}'\le K_{R,T,n,s}\exp\LC-\epsilon^{-\frac{n}{(2n+1)\alpha}}\RC
		\end{equation}
		for some positive constant $K_{R,T,n,s}$. 
	\end{thm}

For $1$-dimensional case ($n=1$), we can also establish the same estimate.

\begin{thm}[Exponential instability II] \label{thm:1dim-main-instability}
	For $n=1$, Theorem~{\rm \ref{thm:main-instability}} is also valid
	with the norm $\|\cdot\|_{*}'$ being replaced by the following norm: 
	\[
	\|\mathcal{A}\|_{*}'':=\sup_{0\not\equiv\chi\in C_{c}^{\infty}((0,T))}\frac{\sup_{t\in(0,T)}\|\chi\mathcal{A}\chi\|_{L^{2}((2,3))\rightarrow L^{2}((2,3))}(t)}{\|\chi\|^2_{W^{2,\infty}(0,T)}}.
	\]
\end{thm}
	
For the local counterpart, let us consider the following initial boundary value
problem for the local wave equation:
\begin{align}\label{eq:local-wave-equation}
	\begin{cases}
		\LC \partial_{t}^{2}-\Delta+q(x)\RC u=0 & \text{ in }\Omega_{T},\\
		\partial_{\nu}u(x,t)=g(x,t) & \text{ in }(\partial\Omega)_{T},\\
		u=\partial_{t}u=0 & \text{ in }\Omega\times\{0\},
	\end{cases}
\end{align}
where $q=q(x)\in L^{\infty}(\Omega)$. It is known that \eqref{eq:local-wave-equation} is well-posed (for example, see \cite{Evan}) with suitable compatibility conditions. Assuming the well-posedness of \eqref{eq:local-wave-equation}, the corresponding (hyperbolic) \emph{Neumann-to-Dirichlet
	map} of \eqref{eq:local-wave-equation} is defined by 
\[
\wt\Lambda_{q}g:=u|_{\partial\Omega\times[0,T]}\quad\text{ for all }g\in C_{c}^{\infty}((\partial\Omega)_{T}).
\]
In fact, $\wt\Lambda_{q}:L^{2}(\partial\Omega\times(0,T))\rightarrow H^{1}(0,T;L^{2}(\partial\Omega))$
is a bounded linear operator, which can be proved by the energy estimate
of \eqref{eq:local-wave-equation}, see e.g. \cite[Section 6.7.5]{CP82PDE}.
Now we assume 
\begin{equation}
	T>{\rm diam}\,(\Omega).\label{eq:propagation-time}
\end{equation}
Under assumption \eqref{eq:propagation-time}, in \cite{RW88uniqueness},
they showed the global uniqueness result for time-independent potentials: 
\[
\wt\Lambda_{q_{1}}=\wt\Lambda_{q_{2}}\quad\text{implies}\quad q_{1}=q_{2} \text{ in }\Omega.
\]
In \cite{Sun90stability}, the author showed that, if \eqref{eq:propagation-time}
holds, under some apriori assumptions, the following estimate hold:
\begin{equation}\label{eq:Holder-stability-wave}
\|q_{1}-q_{2}\|_{L^{2}(\Omega)}\le C\left\|\wt\Lambda_{q_{1}}-\wt\Lambda_{q_{2}}\right\|_{\mathcal{L}}^{\alpha} 
\end{equation}
for some constants $C$ and $\alpha$, where $\norm{\cdot}_{\mathcal{L}}$ stands for the operator norm for the Neumann-to-Dirichlet map. A similar estimate also holds for the hyperbolic Dirichlet-to-Neumann map \cite{alessandrini1990stability}. In other words, the stability of the
inverse problem for the local wave equation is of \emph{H\"{o}lder-type}. We also mention other related results of inverse problems for the local wave equation with potentials \cite{eskin2006new,eskin2007inverse,isakov1991completeness,kian2017unique,ramm1991inverse,salazar2013determination}.

Similar to the local version, we can prove the
global uniqueness result for time-independent potentials for the fractional wave equation (see Theorem~\ref{thm:global uniqueness}).
However, in the nonlocal counterpart of \eqref{eq:Holder-stability-wave}, we 
show that the stability of the inverse problem for the fractional wave equation is of (optimal) \emph{logarithmic-type} in view of Theorem~\ref{thm:main-stability} and Theorem~\ref{thm:main-instability}. 
We also want to point out that we do not need to assume the \emph{large influence time condition} \eqref{eq:propagation-time}. One possible explanation is that while the \emph{speed of propagation} of the local wave equation is \emph{finite}, the speed of propagation of the fractional wave operator is \emph{infinite} due the nonlocal nature of the fractional Laplacian $(-\Delta)^s$, for $0<s<1$. We will offer some detailed arguments in Section \ref{sec:forward}.

Before ending this section, we would like to discuss some interesting results for the time-harmonic wave equation. Consider the time-harmonic wave equation with a potential
	(a.k.a. Schr\"{o}dinger equation): 
	\begin{equation}
		\LC -\Delta+q(x)-\kappa^{2}\RC v=0\quad\text{ in }\;\Omega.\label{eq:Schrodinger-frequency}
	\end{equation}
	Ignoring the effect of the frequency $\kappa>0$, Alessandrini \cite{Ale88stable} proved the well-known logarithmic stability estimate for the inverse boundary value problem of \eqref{eq:Schrodinger-frequency},
	and Mandache \cite{Man01instability} established that this logarithmic
	estimate is optimal by showing that the inverse problem is exponentially
	unstable. Nonetheless, by taking the frequency into account, it was shown
	in \cite{INUW14increasingstability} that 
	\begin{equation}
		\|q_{1}-q_{2}\|_{H^{-\alpha}(\mathbb{R}^{n})}\le C\LC \kappa+\log\frac{1}{{\rm dist}\,(\mathcal{C}_{q_{1}},\mathcal{C}_{q_{2}})}\RC^{-2\alpha-n}+C\kappa^{4}{\rm dist}\LC \mathcal{C}_{q_{1}},\mathcal{C}_{q_{2}}\RC,\label{eq:increasing-stability}
	\end{equation}
	where $\mathcal{C}_{q_1}, \mathcal{C}_{q_2}$ are the Cauchy data of the Schr\"{o}dinger
	equation \eqref{eq:Schrodinger-frequency} corresponding to $q_1, q_2$, and ${\rm dist}\,\LC \mathcal{C}_{q_{1}},\mathcal{C}_{q_{2}}\RC$
	is the Hausdorff distance between $\mathcal{C}_{q_1}$ and $\mathcal{C}_{q_2}$. Isakov \cite{Isa11increasingstability} proved
	a similar estimate in terms of the DN maps.

	The estimate \eqref{eq:increasing-stability} is shown to be optimal in the
	recent paper \cite{KUW21instability}. The estimate \eqref{eq:increasing-stability}
	clearly indicates that the logarithmic part decreases as the frequency
	$\kappa>0$ increases and the estimate changes from a logarithmic
	type to a H\"{o}lder type. This phenomena is termed as the \emph{increasing
		stability}. It is interesting to compare the stability estimate \eqref{eq:increasing-stability}
	of the time-harmonic wave equation \eqref{eq:Schrodinger-frequency}
	with the stability estimate \eqref{eq:Holder-stability-wave} of the local
	wave equation \eqref{eq:local-wave-equation}. 
	
	Similarly to the local wave equation, we consider the following time-harmonic fractional wave equation
	\begin{equation}\label{eq:fractional-Schrodinger}
		\LC (-\Delta)^{s}+q(x)-\kappa^{2}\RC v=0\quad\text{ in }\Omega,
	\end{equation} 
    which is a fractional Schr\"odinger equation.
	Without considering the effect of the frequency $\kappa>0$, R\"{u}land and
	Salo \cite{RS20Calderon} obtained a logarithmic type stability estimate for the inverse
	boundary value problem of the time-harmonic fractional wave equation
	\eqref{eq:fractional-Schrodinger} and, in \cite{RS18Instability}, they proved that such logarithmic estimate is optimal by showing the exponential instability phenomenon. These results give rise to a natural question: in the inverse boundary value problem for \eqref{eq:fractional-Schrodinger}, if we take the frequency $\kappa$ into account, does the increasing stability estimate similar to \eqref{eq:increasing-stability} hold? In view of the optimal logarithmic stability
	results in Theorem~\ref{thm:main-stability} and Theorem~\ref{thm:main-instability}, we have a strong reason to believe that the answer to this question is negative.

	The paper is organized as follows. We discuss and prove the well-posedness of the fractional wave equation in Section~\ref{sec:forward} and in Appendix \ref{Appendix}, respectively. We then prove Theorem~\ref{thm:global uniqueness} in Section~\ref{sec:uniqueness}, and prove Theorem~\ref{thm:main-stability} in Section~\ref{sec:stability}. The approach is mainly based on the qualitative and quantitative Runge approximation properties for the fractional wave equation. Finally, we prove Theorem~\ref{thm:main-instability} and Theorem~\ref{thm:1dim-main-instability} in Section~\ref{sec:optimality}.

	\section{The forward problems for the fractional wave equation}\label{sec:forward}
	In this section, we provide all preliminaries that we need in the rest of the paper. Let us first recall (fractional) Sobolev spaces and prove the well-posedness of the fractional wave equation \eqref{fractional wave main}.
	
	\subsection{Sobolev spaces}\label{subsec:Sobolev}
	Let $\mathcal{F},\ \mathcal{F}^{-1}$ be Fourier transform and its inverse, respectively.
	For $s\in(0,1)$, the fractional Laplacian is defined via 
		\[
		(-\Delta)^{s}u:=\mathcal{F}^{-1}\left(|\xi|^{2s}\mathcal{F}(u)\right), \text{ for }u\in H^s (\R^n),
		\]
		where $H^s(\R^n)$ stands for the $L^2$-based fractional Sobolev space (see \cite{di2012hitchhiks,kwasnicki2017ten,stein2016singular}). 
	The space $H^{a}(\R^{n})=W^{a,2}(\R^{n})$ denotes the (fractional) Sobolev space equipped with the norm 
\[
\|u\|_{H^{a}(\R^{n})}:=\left\Vert \mathcal{F}^{-1}\left\{\left\langle \xi\right\rangle ^{a}\mathcal{F}u\right\}\right\Vert _{L^{2}(\R^{n})},
\]
for any $a\in \R$, where $\left\langle \xi\right\rangle =(1+|\xi|^{2})^{\frac{1}{2}}$.
It is known that for $s\in(0,1)$, $\|\cdot\|_{H^{s}(\R^{n})}$
has the following equivalent representation 
\begin{equation*}
	\|u\|_{H^{s}(\mathbb{R}^{n})}:=\|u\|_{L^{2}(\R^{n})}+[u]_{H^{s}(\R^{n})}\label{eq:NormHs}
\end{equation*}
where 
\[
[u]_{H^{s}(\mathcal{O})}^{2}:=\int_{\mathcal{O}\times\mathcal{O}}\frac{\left|u(x)-u(y)\right|^{2}}{|x-y|^{n+2s}}\, dxdy,
\]
for any open set $\mathcal{O}\subset \R^{n}$.

Given any open set $\mathcal{O}$ of $\mathbb{R}^{n}$ and $a\in\mathbb{R}$,
let us define the following Sobolev spaces, 
\begin{align*}
	H^{a}(\mathcal{O}) & :=\{u|_{\mathcal{O}};\,u\in H^{a}(\R^{n})\},\\
    \wt H^{a}(\mathcal{O}) & :=\text{closure of \ensuremath{C_{c}^{\infty}(\mathcal{O})} in \ensuremath{H^{a}(\R^{n})}},\\
	H_{0}^{a}(\mathcal{O}) & :=\text{closure of \ensuremath{C_{c}^{\infty}(\mathcal{O})} in \ensuremath{H^{a}(\mathcal{O})}},
\end{align*}
and 
\[
H_{\overline{\mathcal{O}}}^{a}:=\{u\in H^{a}(\R^{n});\,\mathrm{supp}(u)\subset\overline{\Omega}\}.
\]
In addition, the Sobolev space $H^{a}(\mathcal{O})$ is complete under the norm
\[
\|u\|_{H^{a}(\mathcal{O})}:=\inf\left\{ \|v\|_{H^{a}(\mathbb{R}^{n})};v\in H^{a}(\mathbb{R}^{n})\mbox{ and }v|_{\mathcal{O}}=u\right\} .
\]
It is not hard to see that $\widetilde{H}^{a}(\mathcal{O})\subseteq H_{0}^{a}(\mathcal{O})$,
and that $H_{\overline{\mathcal{O}}}^{a}$ is a closed subspace of
$H^{a}(\R^{n})$.
We also denote $H^{-s}(\mathcal{O})$ to be the dual space of $\wt H^s(\mathcal{O})$. In fact, $H^{-s}(\mathcal{O})$ has the following characterization: 
\[
H^{-s}(\mathcal{O}) = \begin{Bmatrix} u|_{\mathcal{O}} : u \in H^{-s}(\mathbb{R}^{n}) \end{Bmatrix} \quad \text{with} \quad \inf_{w\in H^{s}(\mathbb{R}^{n}),w|_{\mathcal{O}}=u} \| w \|_{H^{s}(\mathbb{R}^{n})},
\]
see e.g. \cite[Section 2.1]{ghosh2016calder}, \cite[Chapter 3]{mclean2000strongly}, or \cite{triebel2002function} for more details about the fractional Sobolev spaces. Moreover, we will use 
 \begin{align*}
	\LC f, g \RC_{L^2(A)}:=\int_A f g \, dx, \quad \LC F, G \RC _{L^2(A_T)}:=\int_0^T \int_{A} FG \, dxdt,
\end{align*}
in the rest of this paper, for any set $A\subset \R^n$.

\subsection{The forward problem}
	
	We first state the well-posedness of the fractional wave equation.
	As above, let $\Omega\subset \R^n$ be a bounded Lipschitz domain with $n\in \N$. Given $T>0$, $s\in(0,1)$, and $q=q(x)\in L^\infty(\Omega)$, consider the initial exterior value problem for the fractional wave equation
	\begin{align}\label{fractional wave well-posedness}
		\begin{cases}
			\LC \p_t^2 + (-\Delta)^s +q \RC u=F & \text{ in }\Omega_T ,\\
			u=f  & \text{ in }(\Omega_e)_T,\\
			u=\varphi, \quad \p _t u=\psi & \text{ in }\R^n \times \{0\},
		\end{cases}
	\end{align}
	where $f\in C^\infty_c(W_T)$ for some open set with Lipschitz boundary $W\subset \Omega_e$ satisfying $\overline{W}\cap \overline{\Omega}=\emptyset$, $\varphi\in\wt H^{s}(\Omega)$, and $\psi\in L^{2}(\mathbb{R}^{n})$ with $\rm{supp}\,(\psi) \subset \Omega$. We want to show the well-posedness of \eqref{fractional wave well-posedness}. Setting $v:=u-f$, we then consider the fractional wave equation with zero exterior data 
	\begin{align}\label{fractional wave zero exterior}
		\begin{cases}
			\LC \p_t^2 + (-\Delta)^s +q \RC v=\wt F & \text{ in }\Omega_T ,\\
			v=0 & \text{ in }(\Omega_e)_T,\\
			v=\wt \varphi, \quad \p _t v=\wt \psi & \text{ in }\R^n \times \{0\},
		\end{cases}
	\end{align}  
    where $\wt F:=F-(-\Delta)^s f$, $\wt \varphi (x)= \varphi(x)- f(x,0)=\varphi(x)$ and $\widetilde{\psi}(x)=\psi(x)-\p_t f(x,0)=\psi (x)$. Hence, we simply denote the initial data as $(\varphi,\psi)$ in the rest of the paper. Now it suffices to study the well-posedness of \eqref{fractional wave well-posedness}.
	
		Let us introduce the following notations. Define 
	$$
	\bm{u}:[0,T]\to \wt H^s(\Omega)
	$$	
	by 
	\[
	[\bm{u}(t)](x):=u(x,t), \text{ for }x\in \R^n, t\in [0,T].
	\]
	Similarly, the function $\bm{\wt F}:[0,T]\to L^2(\Omega)$ can be defined analogously by
	\[
	[\bm{\wt F}(t)](x):=\wt F(x,t), \text{ for }x\in \R^n, \ t\in [0,T].
	\]
	With these notations at hand, we can define the weak formulation for the fractional wave equation. Let $\phi\in \wt H^s(\Omega)$ be any test function, multiplying \eqref{fractional wave well-posedness} with $\phi$ gives 
	\begin{align*}
		\LC \bm{v}'', \phi \RC_{L^2(\Omega)} + B[\bm{v},\phi;t]=(\bm{\wt F},\phi)_{L^2(\Omega)}, \text{ for }0\leq t\leq T,
	\end{align*}
	where  $B[\bm{v},\phi;t]$ is the bilinear form defined via
	\begin{align*}
		B[\bm{v}, \phi ;t]:=\int_{\R^n} (-\Delta)^{s/2} \bm{v} (-\Delta)^{s/2}\phi  \, dx +\int_{\Omega} q \bm{v}\phi  \, dx.
	\end{align*}

	\begin{defi}[Weak solutions]\label{def:weak-soln}
		A function 
		\[
		\bm{v}\in L^2(0,T; \wt H^s(\Omega)), \text{ with }\bm{v}'\in L^2(0,T;L^2(\Omega)) \text{ and } \bm{v}''\in L^2(0,T;H^{-s}(\Omega))
		\]
		is a \emph{weak solution} of the initial exterior value problem \eqref{fractional wave zero exterior} if 
		\begin{itemize}
			\item[(1)] $(\bm{v}''(t),\phi)_{L^2(\Omega)}+B[\bm{v},\phi;t]=\LC \bm{\wt F},\phi\RC _{L^2(\Omega)}$, for all $\phi\in \wt H^s(\Omega)$, and for  $0\leq t \leq T$ a.e.
			
			\item[(2)] $\bm{v}(0)= \wt\varphi$ and $\bm{v}'(0)= \wt\psi$.
		\end{itemize}
	\end{defi}

	\begin{thm}[Well-posedness] \label{thm:well-posedness}
		For any $\bm{\wt F}\in L^2(0,T;L^{2}(\Omega))$, $\wt\varphi\in \wt H^{s}(\Omega)$, and $\wt\psi \in L^{2}(\mathbb{R}^{n})$ with $\rm{supp}\,(\wt\psi)\subset \Omega$, there exists a unique weak solution $\bm{v}$ to \eqref{fractional wave zero exterior}. Moreover, the following estimate holds:
		\begin{align}
			\begin{split}
				& \|\bm{v}\|_{L^{\infty}(0,T;\wt {H}^{s}(\Omega))}+\|\partial_{t}\bm{v}\|_{L^{\infty}(0,T;L^{2}(\Omega))} \\
				& \qquad  \le C \LC \|\bm{\wt F}\|_{L^{2}(0,T;L^{2}(\Omega))}+\|\wt\varphi\|_{\wt{H}^{s}(\Omega)}+\|\wt\psi\|_{L^{2}(\Omega)}\RC.\label{eq:energy-est0}
			\end{split}
		\end{align}
	\end{thm}

	\begin{cor}\label{thm:well}
		Let $\Omega\subset \R^n$ be a bounded Lipschitz domain for $n\in \N$, and $W\subset\Omega_e$ be any open set with Lipschitz boundary satisfying $\overline{W}\cap \overline{\Omega}=\emptyset$. Then for any $F=F(x,t)\in L^2(0,T;L^{2}(\Omega))$, $f=f(x,t)\in C_{c}^{\infty}(W_T)$, $\varphi\in \wt H^{s}(\Omega)$, and $\psi \in L^{2}(\mathbb{R}^{n})$ with $\rm{supp}\,(\psi)\subset \Omega$, there exists a unique weak solution $u=v+f$ of \eqref{fractional wave well-posedness}, where $v\in L^2(0,T;\wt H^s(\Omega)) \cap H^1(0,T;L^2(\Omega))$ is the unique weak solution of \eqref{fractional wave zero exterior}. Furthermore, we have the following estimate
		\begin{align}\label{eq:energy-est}
			\begin{split}
				& \|u-f\|_{L^{\infty}(0,T;\wt {H}^{s}(\Omega))}+\|\partial_{t}(u - f)\|_{L^{\infty}(0,T;L^{2}(\Omega))} \\
				& \qquad  \le C \LC \|F-(-\Delta)^{s}f\|_{L^{2}(0,T;L^{2}(\Omega))}+\|\varphi\|_{\wt{H}^{s}(\Omega)}+\|\psi\|_{L^{2}(\Omega)}\RC.
			\end{split}
		\end{align}
	\end{cor}
	    
    The proof of Theorem \ref{thm:well-posedness} is similar to the well-posedness of the classical wave equation (i.e., $s=1$) and, for the sake of completeness, we will give a comprehensive proof in Appendix \ref{Appendix}. In this article, we only consider the time-independent potential $q=q(x)\in L^\infty(\Omega)$. In fact, the well-posedness for a space-time dependent potential $q=q(x,t)\in L^\infty(\Omega_{T})$ has been studied. We refer to \cite[Theorem 10.14]{Bre11} for the well-posedness of the abstract wave equations, and to \cite{DFA20nonlocalwave} for the well-posedness result for non-local semi-linear integro-differential wave equations which involve both the fractional Laplacian (in space) and the Caputo fractional derivative operator (in time).  
   
\subsection{The DN map and its duality}
  With the well-posedness at hand, one can define the corresponding DN map \eqref{eq:hyperbolic-DN} for the fractional wave equation \eqref{fractional wave main}. Let us define the solution operator 
  \begin{align}\label{Poisson operator}
  	\mathcal{P}_{q} : C^\infty_c (W_T) \to L^2 (0,T;H^s(\Omega)), \quad f\mapsto u|_{\Omega_T},
  \end{align}
  where $W\subset \Omega_e$ is a Lipschitz set with $\overline{W}\cap \overline{\Omega}=\emptyset$, and $u$ is the solution of \eqref{fractional wave main}. Given any $\varphi(x,t)$ defined in $(\Omega_{{\rm e}})_{T}$, we define 
  \[
  \varphi^\ast (x,t):=\varphi(x,T-t)\quad\text{for all}\quad (x,t)\in(\Omega_{{\rm e}})_{T},
  \]
  and we define the following backward DN-map: 
  \[
  \Lambda^\ast _{q}(f):=(\Lambda_{q}(f))^\ast \quad\text{for any }\,\,f\in C_{c}^{\infty}((\Omega_{{\rm e}})_{T}).
  \]

  \begin{lem}
   Given any $q\in L^{\infty}(\Omega)$, $\Lambda^\ast _{q}$ is self-adjoint, that is, 
  	\begin{equation*}
  		\int_{(\Omega_{e})_{T}}\Lambda^\ast_{q}(f_{1})f_{2}\, dxdt=\int_{(\Omega_{e})_{T}}f_{1}\Lambda^\ast_{q}(f_{2})\, dxdt,\quad\text{ for all }f_{1},f_{2}\in C_{c}^{\infty}((\Omega_{{\rm e}})_{T}).\label{eq:self-adjoint}
  	\end{equation*}
  \end{lem}
  
  \begin{proof}
  	Let $u_{1}=\mathcal{P}_{q}f_{1}$ and $u_{2}=\mathcal{P}_{q}f_{2}$. Using integration by parts, we have 
  	\begin{equation}
  		\int_{\Omega_{T}}\left[u_{1}(\partial_{t}^{2}u^\ast_{2})-(\partial_{t}^{2}u_{1})u^\ast_{2}\right] dxdt=0.\label{eq:integration-by-parts}
  	\end{equation}
  	Therefore, 
  	\begin{align*}
  		0 & =\int_{\Omega_{T}}\left[u_{1}\LC \partial_{t}^{2}u^\ast_{2}+(-\Delta)^{s}u^\ast_{2}+qu^\ast_{2}\RC-\LC\partial_{t}^{2}u_{1}+(-\Delta)^{s}u_{1}+q(x)u_{1}\RC u^\ast_{2}\right] dxdt\\
  		& =\int_{\Omega_{T}}\left[u_{1}((-\Delta)^{s}u^\ast_{2})-((-\Delta)^{s}u_{1})u^\ast_{2}\right] dxdt\\
  		& =\LC \int_{(\mathbb{R}^{n})_{T}}-\int_{(\Omega_{{\rm e}})_{T}}\RC\left[u_{1}((-\Delta)^{s}u^\ast_{2})-((-\Delta)^{s}u_{1})u^\ast_{2}\right] dxdt\\
  		& =-\int_{(\Omega_{{\rm e}})_{T}}\left[u_{1}((-\Delta)^{s}u^\ast_{2})-((-\Delta)^{s}u_{1})u^\ast_{2}\right] dxdt\\
  		& =-\int_{(\Omega_{{\rm e}})_{T}}\left[f_{1}\Lambda^\ast_{q}(f_{2})-\Lambda_{q}(f_{1})f^\ast_{2}\right]dxdt.
  	\end{align*}
  	Finally, changing the variable $t\mapsto T-t$, we have 
  	\begin{equation}
  		\int_{(\Omega_{{\rm e}})_{T}}f_{1}\Lambda^\ast_{q}(f_{2}) \, dxdt=\int_{(\Omega_{{\rm e}})_{T}}\Lambda_{q}(f_{1})f^\ast_{2}\, dxdt=\int_{(\Omega_{{\rm e}})_{T}}\Lambda^\ast_{q}(f_{1})f_{2}\, dxdt,\label{eq:self-adjoint2}
  	\end{equation}
  	which is our desired lemma. 
  \end{proof}

Since $\Lambda^\ast_{q}$ is self-adjoint, we can derive the following identity immediately.

  \begin{lem}[Integral identity]\label{lem:self-adjoint}
  		  	Let $q_{1},q_{2}\in L^{\infty}(\Omega)$, and given any $f_{1},f_{2}\in  C_{c}^{\infty}((\Omega_{e})_{T})$. Let $u_{1}:=\mathcal{P}_{q_{1}}f_{1}$ and $u_{2}:=\mathcal{P}_{q_{2}}f_{2}$, where the operator $\mathcal{P}_q$ is given in \eqref{Poisson operator}, for $q=q_1$ and $q=q_2$, respectively. Then 
  	\begin{equation}\label{eq:crucial-identity}
  		\int_{\Omega_{T}}(q_{1}-q_{2})u_{1}u^\ast_{2} \, dxdt=\int_{(\Omega_{{\rm e}})_T}((\Lambda_{q_{1}}-\Lambda_{q_{2}})f_{1})f^\ast_{2}\, dxdt.
  	\end{equation}
  \end{lem}
  
  \begin{proof}
  	Using \eqref{eq:integration-by-parts}, we have 
  	\begin{align*}
  		& \int_{\Omega_{T}}(q_{1}-q_{2})u_{1}u^\ast_{2} \, dxdt\\
  		 =&\int_{\Omega_{T}}\left[q_{1}u_{1}u^\ast_{2}-u_{1}(q_{2}u^\ast_{2})\right] dxdt\\
  		 =&-\int_{\Omega_{T}}\left[(\partial_{t}^{2}u_{1}+(-\Delta)^{s}u_{1})u_2^\ast-u_{1}(\partial_{t}^{2}u^\ast_{2}+(-\Delta)^{s}u^\ast_{2})\right] dxdt\\
  		=&-\int_{\Omega_{T}}\left[((-\Delta)^{s}u_{1})u^\ast_{2}-u_{1}((-\Delta)^{s}u^\ast_{2})\right] dxdt\\
  		 =&\LC \int_{(\Omega_{e})_{T}}-\int_{(\mathbb{R}^{n})_{T}}\RC \left[((-\Delta)^{s}u_{1})u^\ast_{2}-u_{1}((-\Delta)^{s}u^\ast_{2})\right] dxdt\\
  		=&\int_{(\Omega_{e})_{T}}\left[((-\Delta)^{s}u_{1})u^\ast_{2}-u_{1}((-\Delta)^{s}u^\ast_{2})\right] dxdt\\
  		 =&\int_{(\Omega_{e})_{T}}\left[\Lambda_{q_{1}}(f_{1})f^\ast_{2}-f_{1}\Lambda^\ast_{q_{2}}(f_{2})\right] dxdt.
  	\end{align*}
  	Combining with \eqref{eq:self-adjoint2}, we obtain
  	\[
  	\int_{\Omega_{T}}(q_{1}-q_{2})u_{1}u^\ast_{2}\, dxdt=\int_{(\Omega_{{\rm e}})_{T}}\left[(\Lambda_{q_{1}}f_{1})f^\ast_{2}-(\Lambda_{q_{2}}f_{1})f^\ast_{2}\right] dxdt,
  	\]
  	which is our desired lemma. 
  \end{proof}

	\section{Global uniqueness for the fractional wave equation\label{sec:uniqueness}}

	In this section, let us state and prove a qualitative Runge type approximation for the fractional wave equation, and then prove Theorem \ref{thm:global uniqueness}.
	Before further discussion, let us comment on the speeds of propagation of the local and nonlocal wave equations. Given $V=V(x)\in L^\infty (\R^n)$, let $u$ be a solution of 
	$$
	\LC \p_t^2  -\Delta +V\RC u =0 \text{ in } \R^n\times (0,\infty).
	$$ 
	It is known that if $u(x,0)=\phi(x)$, for $x\in \R^n$, such that $\phi \not \equiv 0$ and $\phi$ is compactly supported, then for every $t>0$, the solution $u(\cdot, t)$ has compact support. 
	
	On the other hand, the speed of propagation for the fractional wave equation is infinite due to the nonlocal nature of the fractional Laplacian. To prove this rigorously, let us recall the strong uniqueness property for the fractional Laplacian. Given $0<s<1$, $r\in \R$, if $u\in H^{r}(\R^n)$ satisfies $u=(-\Delta)^ s u =0$ in any nonempty open subset of $\R^n$, then $u\equiv 0$ in $\R^n$. By this property, we can prove the following lemma. 
	
\begin{lem}
	Given $V=V(x)\in L^\infty (\R^n)$, let $u$ be a solution of 
	\begin{equation}\label{eq:nonlocal-wave}
		\LC \partial_{t}^{2}+(-\Delta)^{s}+V\RC u=0\quad\text{ in }\;\mathbb{R}^{n}\times(0,T),
	\end{equation}
then $u$ does not have a finite speed of propagation.
\end{lem}

\begin{proof}
Suppose the contrary, that the speed of propagation of \eqref{eq:nonlocal-wave} is finite. If we choose $u(x,0)=\phi(x)$ for some $0\not\equiv\phi\in C_{c}^{\infty}(\mathbb{R}^{n})$,
given any $T>0$, there exists a bounded set $\Omega$ such that 
\begin{equation*}
	u=0\quad\text{ in }(\Omega_{e})_{T},\label{eq:speed1}
\end{equation*}
therefore, $\partial_{t}^{2}u=0$ in $(\Omega_{e})_{T}$. Using \eqref{eq:nonlocal-wave},
we also have 
\begin{equation*}
	(-\Delta)^{s}u=0\quad\text{ in }(\Omega_{e})_{T}.\label{eq:speed2}
\end{equation*}
Using the strong uniqueness for the fractional Laplacian, we conclude that $u\equiv0$,
which implies $\phi\equiv0$, this is a contradiction. 
\end{proof}

	\subsection{Qualitative Runge approximation}
	The qualitative approximation property is based on the strong uniqueness for the fractional Laplacian (\cite[Theorem 1.2]{ghosh2016calder}).
	\begin{thm}[Qualitative Runge approximation]\label{thm:Runge}
	 Let $\Omega\subset \R^n$ be a bounded Lipschitz domain for $n\in \N$, and $W\subset\Omega_e$ be an open set with Lipschitz boundary satisfying $\overline{W}\cap \overline{\Omega}=\emptyset$. For $s\in (0,1)$, let $\mathcal{P}_{q}$ be the solution operator given by \eqref{Poisson operator}, and define 
	 \begin{align*}
	 	\mathcal{D}:=\left\{ u|_{\Omega_T} : \  u=\mathcal{P}_{q} f, \ f\in C^\infty_c (W_T) \right\}.
	 \end{align*}
    Then $\mathcal{D}$ is dense in $L^2(\Omega_T)$.
	\end{thm}
	
	\begin{rmk}
		The Runge approximation plays an essential role in the study of fractional inverse problems, for example, see \cite{ghosh2016calder,ghosh2017calder,RS20Calderon,cekic2020calderon} and references therein.
	\end{rmk}
	
	\begin{proof}[Proof of Theorem~{\rm \ref{thm:Runge}}]
		By using the Hahn-Banach theorem and the duality arguments, it suffices to show that if $v\in L^2 (\Omega_T)$, which satisfies 
		\begin{align}\label{orthogonal condition in Runge}
			\LC \mathcal{P}_{q} f , v\RC _{L^2(\Omega_T)}=0, \quad \text{ for any }f\in C_{c}^{\infty}(W_T),
		\end{align}
	then $v\equiv 0 $ in $\Omega_T$. Now, consider the adjoint wave equation 
	\begin{align}\label{adjoint fractional wave equation}
		\begin{cases}
			\LC \p_t^2 +(-\Delta)^s+ q \RC w= v &\text{ in }\;\Omega_T, \\
			w=0 &\text{ in }\;(\Omega_e)_T,\\
			w=\p_t w=0 &\text{ in }\; \R^n\times \{T\}.
		\end{cases}
	\end{align}
    Similar to the proof of Theorem \ref{thm:well-posedness}, it is easy to see that \eqref{adjoint fractional wave equation} is well-posed.

    For $f\in C_{c}^{\infty}(W_T)$, let $u$ and $w$ be the solutions of \eqref{fractional wave main} and \eqref{adjoint fractional wave equation}, respectively. Note that $u-f$ is only supported in $\overline{\Omega}_T$, then we have 
    \begin{align}\label{Runge computation}
    	\begin{split}
    	  \LC \mathcal{P}_{q} f, v \RC_{L^2(\Omega_T)}= &\LC u-f, (-\p_t^2 +(-\Delta)^s +q)w \RC_{L^2(\Omega_T)} \\
    	  =& -\LC f, (-\Delta)^s w \RC _{L^2(W_T)},
    	\end{split}
    \end{align}
where we have used $u$ is the solution of \eqref{fractional wave main}, 
$
u(x,0)=\p_t u(x,0)=0 $ and $w(x,T)=\p_t w(x,T)=0$ for $x\in \R^n$ in last equality of \eqref{Runge computation} . By using the conditions \eqref{orthogonal condition in Runge} and \eqref{Runge computation}, one must have $\LC f, (-\Delta)^s w \RC _{L^2(W_T)}=0$, for any $f\in  C_{c}^{\infty}(W_T)$, which implies that 
\begin{align*}
	w = (-\Delta)^s w=0 \text{ in }W_T.
\end{align*}
Fix any fixed $t\in(0,T)$, the strong uniqueness for the fractional Laplacian (see \cite[Theorem 1.2]{ghosh2016calder}) yields that $w (\cdot,t)=0$ in $\R^n\times \{t\}$, for all $t\in (0,T)$. Therefore, we derive $v=0$ as desired, and the Hahn-Banach theorem infers the density property. This proves the assertion.
	\end{proof}
	
	\begin{rmk}
		By using similar arguments, one can also consider the well-posedness (Theorem~{\rm \ref{thm:well-posedness}}) 
		and the Runge approximation (Theorem~{\rm \ref{thm:Runge}}) also hold for the case $q=q(x,t)\in L^\infty(\Omega_T)$. In this work, we are only interested in time-independent potentials $q=q(x)$.
	\end{rmk}

	\begin{rmk}
		For other unique continuation property for the fractional elliptic operators, we refer the reader to \cite{FF14unique,GR19unique,Rul15unique,Yu17unique} and references therein. 
	\end{rmk}

\subsection{Proof of Theorem \ref{thm:global uniqueness}}

With the help of Lemma~\ref{lem:self-adjoint} and Theorem~\ref{thm:Runge}, we can prove the global uniqueness of the inverse problem for the fractional wave equation.

\begin{proof}[Proof of Theorerm~{\rm \ref{thm:global uniqueness}}]
Given any $g\in L^{2}(\Omega_{T})$, using Theorem~\ref{thm:Runge}, there exists a sequence $f_{1,k}\in C_{c}^{\infty}((W_{1})_{T})$ such that 
\[
\lim_{k\rightarrow\infty}\left\|u_{1,k}-g\right\|_{L^{2}((0,T)\times\Omega)}=0,\quad\text{where}\quad u_{1,k}=\mathcal{P}_{q}f_{1,k}.
\]
Since $1\in L^{2}(\Omega_{T})$, similarly, we can choose a sequence $f_{2,k}\in C_{c}^{\infty}((W_{2})_{T})$ such that 
\[
\lim_{k\rightarrow\infty}\left\|u^\ast_{2,k}-1\right\|_{L^{2}((0,T)\times\Omega)}=0,\quad\text{where}\quad u_{2,k}=\mathcal{P}_{q}f_{2,k}.
\]
Combining \eqref{eq:assump} and \eqref{eq:crucial-identity}, we know that 
\[
\int_{\Omega_{T}}(q_{1}-q_{2})u_{1,k}u^\ast_{2,k} \, dxdt=0.
\]
Taking the limit $k\rightarrow\infty$, we obtain 
\[
\int_{\Omega_{T}}(q_{1}-q_{2})g \, dxdt=0.
\]
Finally, by the arbitrariness of $g\in L^{2}(\Omega_{T})$, we conclude that $q_{1}= q_{2}$ in $\Omega_T$. 
\end{proof}

\section{Stability for the fractional wave equation\label{sec:stability}}

In order to understand the stability estimate for the fractional wave equation, let us recall the famous \emph{Caffarelli-Silvestre extension} \cite{caffarelli2007extension} for the fractional Laplacian. 
For each $x'\in\mathbb{R}^{n}$ and $x_{n+1}\in\mathbb{R}_{+}^{n+1}$,
we denote $x=(x',x_{n+1})\in\mathbb{R}^{n}\times\mathbb{R}_{+}=\mathbb{R}_{+}^{n+1}$.
Fixing any $0<s<1$ and $t\in(0,T)$. If there exists $\gamma\in\mathbb{R}$
such that $\bm{v}(t)=v(x',t)\in H^{\gamma}(\mathbb{R}^{n})$, using
\cite[Lemma 4.1]{RS20Calderon}, there exists a \textit{Caffarelli-Silvestre
	extension} $\bm{v}^{{\rm cs}}(t)=v^{{\rm cs}}(x',x_{n+1},t)\in C^{\infty}(\mathbb{R}_{+}^{n+1})$
of $v$ satisfies 
\begin{subequations}
	\begin{align*}
		\nabla\cdot x_{n+1}^{1-2s}\nabla v^{{\rm cs}} & =0\quad\text{ in }\mathbb{R}_{+}^{n+1},\\
		v^{{\rm cs}} & =v\quad\text{ on }\mathbb{R}^{n}\times\{0\},\\
		\lim_{x_{n+1}\rightarrow0}x_{n+1}^{1-2s}\partial_{n+1}v^{{\rm cs}} & =-a_{n,s}(-\Delta)^{s}v,
	\end{align*}
\end{subequations}
where $a_{n,s}:=2^{1-2s}\frac{\Gamma(1-s)}{\Gamma(s)}$ and $\nabla = (\nabla_{x'},\p_{x_{n+1}})=(\nabla ', \p_{n+1})$.


\subsection{Logarithmic stability of the Caffarelli-Silvestre extension}

We now define 
\[
\hat{\Omega}:=\left\{x\in\mathbb{R}^{n}\times\{0\}: \, {\rm dist}\,(x,\Omega)<\frac{1}{2}{\rm dist}\,(\Omega,W)\right\}.
\]
We now prove a lemma, which concerns the \emph{propagation of smallness} for the Caffarelli-Silvestre extension.  By using similar ideas as in \cite[Section 5]{RS20Calderon}, we can derive the following boundary logarithmic stability estimate.

\begin{lem}\label{lem:small1}
	Let $W\subset\Omega_{e}$ be
	an open bounded Lipschitz set such that $\overline{W}\cap\overline{\Omega}=\emptyset$.
	Let $v^{{\rm cs}}(x',x_{n+1},t)$ be the Caffarelli-Silvestre extension
	of $v(x',t)$. Define
	\[
	\eta(t):=\left\|\lim_{x_{n+1}\rightarrow0}x_{n+1}^{1-2s}\partial_{n+1}\bm{v}^{{\rm cs}}(t)\right\|_{H^{-s}(W)}=a_{s}\|(-\Delta)^{s}\bm{v}(t)\|_{H^{-s}(W)}
	\]
	Suppose that there exist constants $C_{1}>1$ and $E>0$
	such that $\eta(t)\le E$ and 
	\begin{equation}
		\left\|x_{n+1}^{\frac{1-2s}{2}}v^{{\rm cs}}\right\|_{L^{\infty}(0,T;L^{2}(\mathbb{R}^{n}\times[0,C_{1}]))}+\left\|x_{n+1}^{\frac{1-2s}{2}}\nabla v^{{\rm cs}}\right\|_{L^{\infty}(0,T;L^{2}(\mathbb{R}_{+}^{n+1}))}\le E,\label{eq:apriori1b}
	\end{equation}
	then 
	\begin{equation}
		\left\|x_{n+1}^{\frac{1-2s}{2}}\bm{v}^{{\rm cs}}(t)\right\|_{L^{2}(\hat{\Omega}\times[0,1])}\le CE\,\log^{-\mu}\LC\frac{CE}{\eta(t)}\RC\label{eq:stability1}
	\end{equation}
	for some constants $C>1$ and $\mu>0$, both depending only on $n,s,C_{1},\Omega,W$.
	Moreover, given any $\gamma>0$, we have 
	\begin{equation}
		\left\|x_{n+1}^{\frac{1-2s}{2}+\gamma}\nabla\bm{v}^{{\rm cs}}(t)\right\|_{L^{2}(\hat{\Omega}\times[0,1])}\le CE\,\log^{-\mu}\LC\frac{CE}{\eta(t)}\RC,\label{eq:stability2}
	\end{equation}
	for some constants $C>1$ and $\mu>0$, both depending only on $n,s,C_{1},\Omega,W$,
	as well as $\gamma$. 
\end{lem}

\begin{proof}
	Estimate \eqref{eq:stability1} is an immediate consequence of \cite[(5.3) of Theorem 5.1]{RS20Calderon}.
	Replacing \cite[(5.67) of Theorem 5.1]{RS20Calderon} by the following
	inequality: 
	\begin{align*}
		&\left\|x_{n+1}^{\frac{1-2s}{2}+\gamma}\nabla'\bm{v}^{{\rm cs}}(t)\right\|_{L^{2}(\hat{\Omega}\times[0,h])} \\ \le& C\left\|x_{n+1}^{\frac{1-2s}{2}}\nabla'\bm{v}^{{\rm cs}}(t)\right\|_{L^{2}(\hat{\Omega}\times[0,h])} \left\|x_{n+1}^{\gamma}\right\|_{L^{\infty}(\Omega\times[0,h])}\\
		 \le & Ch^{\gamma}E,
	\end{align*}
	and 
	\begin{align*}
		& \left\|x_{n+1}^{\frac{1-2s}{2}+\gamma}\partial_{n+1}\bm{v}^{{\rm cs}}(t)\right\|_{L^{2}(\hat{\Omega}\times[0,h])} \\ 
		\le & C\left\|x_{n+1}^{\frac{1-2s}{2}}\partial_{n+1}\bm{v}^{{\rm cs}}(t)\right\|_{L^{2}(\hat{\Omega}\times[0,h])}\left\|x_{n+1}^{\gamma}\right\|_{L^{\infty}(\Omega\times[0,h])}\\
		\le & Ch^{\gamma}E,
	\end{align*}
	we can prove \eqref{eq:stability2} using the similar argument as
	in the proof of \cite[(5.5), Theorem 5.1]{RS20Calderon}, with a slight
	modification as indicated above. 
\end{proof}
\begin{rmk}
In view of \cite[Lemma 4.2]{RS20Calderon}, we have 
\begin{align*}
&\left\|x_{n+1}^{\frac{1-2s}{2}}v^{{\rm cs}}\right\|_{L^{\infty}(0,T;L^{2}(\mathbb{R}^{n}\times[0,C_{1}]))}+\left\|x_{n+1}^{\frac{1-2s}{2}}\nabla v^{{\rm cs}}\right\|_{L^{\infty}(0,T;L^{2}(\mathbb{R}_{+}^{n+1}))}\\
\le & C\|v\|_{L^{\infty}(0,T;H^{s}(\mathbb{R}^{n}))},
\end{align*}
therefore, \eqref{eq:apriori1b} can be achieved by the following
sufficient condition: 
\[
\norm{v}_{L^{\infty}(0,T;H^{s}(\mathbb{R}^{n}))}\le E.
\]
\end{rmk}

We now define 
\[
\omega_{1}(z):=\log^{-\mu}\LC \frac{C}{z}\RC ,\quad\text{for}\;\,0<z<1.
\]
Note that $\omega_{1}^{2}(z)$ is concave on $z\in(0,z_{0})$ for
some sufficiently small $z_{0}=z_0(\mu)>0$. From
\eqref{eq:stability1} and \eqref{eq:stability2}, together with \cite[Lemma 4.4]{RS20Calderon},
we obtain 
\begin{align}\label{eq:patch1-stability1}
 \begin{split}
 	& \mathbbm1_{\left\{\eta(t)<z_{0}E\right\}}\|\bm{v}(t)\|_{{H}^{s-\gamma}(\Omega)} \\
 	\le&\mathbbm1_{\left\{\eta(t)<z_{0}E\right\}}C \LC \left\| x_{n+1}^{\frac{1-2s}{2}}\bm{v}^{{\rm cs}}(t)\right\|_{L^{2}(\hat{\Omega}\times[0,1])}+\left\| x_{n+1}^{\frac{1-2s}{2}+\gamma}\nabla\bm{v}^{{\rm cs}}(t)\right\|_{L^{2}(\hat{\Omega}\times[0,1])}\RC \\
 	\le&\mathbbm1_{\left\{\eta(t)<z_{0}E\right\}}CE\omega_{1} \LC \frac{\eta(t)}{E}\RC\\
 	=&CE\omega_{1}\LC \mathbbm1_{\left\{\frac{\eta(t)}{E}<z_{0}\right\}}\frac{\eta(t)}{E}\RC.
 \end{split}
\end{align}

Using Jenson's inequality for concave functions, we have 
\begin{equation}
\frac{1}{T}\int_{0}^{T}\omega_{1}^{2}\LC \mathbbm1_{\left\{\frac{\eta(t)}{E}<z_{0}\right\}}\frac{\eta(t)}{E}\RC dt\le\omega_{1}^{2}\LC \frac{1}{T}\int_{0}^{T}\mathbbm1_{\left\{\frac{\eta(t)}{E}<z_{0}\right\}}\frac{\eta(t)}{E}\,dt\RC.\label{eq:patch2-stability1}
\end{equation}
Combining \eqref{eq:patch1-stability1} and \eqref{eq:patch2-stability1} implies 
\begin{equation}\label{eq:patch1-stability2}
\int_{0}^{T}\mathbbm1_{\left\{\eta(t)<z_{0}E\right\}}\|\bm{v}(t)\|_{{H}^{s-\gamma}(\Omega)}^{2}\,dt\le C^{2}E^{2}T\omega_{1}^{2}\LC \frac{1}{TE}\int_{0}^{T}\mathbbm1_{\left\{\frac{\eta(t)}{E}<z_{0}\right\}}\eta(t)\,dt\RC.
\end{equation}
We extend $\omega_{1}$ so that it is continuous and monotone increasing
on $(0,\infty)$. Therefore, \eqref{eq:patch1-stability2} gives
\begin{align}\label{eq:patch1-stability3}
 \begin{split}
 	 \int_{0}^{T}\mathbbm1_{\left\{\eta(t)<z_{0}E\right\}}\|\bm{v}(t)\|_{{H}^{s-\gamma}(\Omega)}^{2}\,dt
 	\le & C^{2}E^{2}T\omega_{1}^{2}\LC \frac{1}{TE}\int_{0}^{T}\eta(t)\,dt\RC \\
 	\le& C^{2}E^{2}T\omega_{1}^{2}\LC \frac{\|\eta\|_{L^{2}(0,T)}}{T^{\frac{1}{2}}E}\RC\\
 	=&C^{2}E^{2}T\omega_{1}^{2}\LC \frac{a_{s}\|(-\Delta)^{s}v\|_{L^{2}(0,T;H^{-s}(W))}}{T^{\frac{1}{2}}E}\RC.
 \end{split}
\end{align}
On the other hand, from \cite[Lemma 4.4]{RS20Calderon}, it follows
that 
\begin{align}\label{eq:patch1-stability4}
\begin{split}
	 & \mathbbm1_{\left\{\eta(t)\ge z_{0}E\right\}}\|\bm{v}(t)\|_{{H}^{s-\gamma}(\Omega)} \\
	\le & \mathbbm1_{\left\{\eta(t)\ge z_{0}E\right\}}C \LC   \left\| x_{n+1}^{\frac{1-2s}{2}}\bm{v}^{{\rm cs}}(t)\right\|_{L^{2}(\hat{\Omega}\times[0,1])}+\left\| x_{n+1}^{\frac{1-2s}{2}+\gamma}\nabla\bm{v}^{{\rm cs}}(t)\right\|_{L^{2}(\hat{\Omega}\times[0,1])}\RC \\
	\le &\mathbbm1_{\left\{\eta(t)\ge z_{0}E\right\}}C \LC \left\| x_{n+1}^{\frac{1-2s}{2}}v^{{\rm cs}}\right\|_{L^{\infty}(0,T;L^{2}(\mathbb{R}^{n}\times[0,C_{1}]))}+\left\| x_{n+1}^{\frac{1-2s}{2}}\nabla v^{{\rm cs}}\right\|_{L^{\infty}(0,T;L^{2}(\mathbb{R}_{+}^{n+1}))}\RC \\
	\le&\mathbbm1_{\left\{\eta(t)\ge z_{0}E\right\}}CE \\
	\le &Cz_{0}^{-1}\eta(t)\\
	= & Cz_{0}^{-1}a_{s}\|(-\Delta)^{s}\bm{v}(t)\|_{H^{-s}(W)}.
\end{split}
\end{align}
Squaring both sides of \eqref{eq:patch1-stability4} and subsequently
integrating it, we obtain 
\begin{align}\label{eq:patch1-stability5}
\begin{split}
	&\int_{0}^{T}\mathbbm1_{\left\{\eta(t)\ge E\right\}}\|\bm{v}(t)\|_{{H}^{s-\gamma}(\Omega)}^{2}\,dt \\ \le &C^{2}z_{0}^{-2}a_{s}^{2}\|(-\Delta)^{s}v\|_{L^{2}(0,T;H^{-s}(W))}^{2} \\
	=&C^{2}z_{0}^{-2}E^{2}a_{s}^{2} \LC \frac{\|(-\Delta)^{s}v\|_{L^{2}(0,T;H^{-s}(W))}}{E}\RC^{2}.
\end{split}
\end{align}
Summing \eqref{eq:patch1-stability3} and \eqref{eq:patch1-stability5} yields 
\begin{align}\label{eq:patch1-stability6}
 \begin{split}
 	 \|v\|_{L^{2}(0,T;{H}^{s-\gamma}(\Omega))}^{2}=&\int_{0}^{T}\|\bm{v}(t)\|_{{H}^{s-\gamma}(\Omega)}^{2}\,dt \\
  \le & C^{2}E^{2}\left[T\omega_{1}^{2}\LC\frac{a_{s}\|(-\Delta)^{s}v\|_{L^{2}(0,T;H^{-s}(W))}}{T^{\frac{1}{2}}E}\RC  \right.\\
  & \qquad \qquad \left. +z_{0}^{-2}a_{s}^{2}\LC \frac{\|(-\Delta)^{s}v\|_{L^{2}(0,T;H^{-s}(W))}}{E}\RC ^{2}\right].
 \end{split}
\end{align}

We now define 
\[
\omega(z):=\LC T\omega_{1}^{2}(T^{-\frac{1}{2}}a_{s}z)+z_{0}^{-2}a_{s}^{2}z^{2}\RC^{\frac{1}{2}}.
\]
Note that $\omega(z)$ is of logarithmic type when $z$ is small. Therefore, \eqref{eq:patch1-stability6} can be written as 
\begin{equation*}
\|v\|_{L^{2}(0,T;{H}^{s-\gamma}(\Omega))}\le CE\omega \LC \frac{\|(-\Delta)^{s}v\|_{L^{2}(0,T;H^{-s}(W))}}{E}\RC.
\end{equation*}
We summarize the above discussions in the following corollary. 
\begin{cor}\label{cor:small2} 	
	Let $W\subset\Omega_{e}$ be an open bounded Lipschitz set and $\overline{W}\cap\overline{\Omega}=\emptyset$.
	If there exists a constant $E>0$ such that 
	\begin{equation}
		\|v\|_{L^{\infty}(0,T;H^{s}(\mathbb{R}^{n}))}\le E,\label{eq:apriori2}
	\end{equation}
	then there exists a constant $C>1$ and a function of logarithmic type $\omega$, both depending only
	on $n,s,\Omega,W,\gamma,T$, such that 
	\begin{equation*}
		\|v\|_{L^{2}(0,T;H^{s-\gamma}(\Omega))}\le CE\omega\LC \frac{ \left\|(-\Delta)^{s}v\right\|_{L^{2}(0,T;H^{-s}(W))}}{E}\RC.
	\end{equation*}
\end{cor}

\subsection{Quantitative unique continuation}

Given any $F\in L^{2}(\Omega_{T})$, by Corollary~\ref{thm:well},
there exists a unique solution $v_{F}$ of the backward wave equation 
\begin{equation}
	\begin{cases}
		\LC \partial_{t}^{2}+(-\Delta)^{s}+q\RC  v_{F}=F & \text{ in }\Omega_{T},\\
		v_{F}=0 & \text{ in }(\Omega_{e})_{T},\\
		v_{F}=\partial_{t}v_{F}=0 & \text{ in }\mathbb{R}^{n}\times\{T\},
	\end{cases}\label{eq:backward-heat-equation}
\end{equation}
such that 
\begin{equation*}
	\|v_{F}\|_{L^{\infty}(0,T;\wt{H}^{s}(\Omega))}\le C_0\|F\|_{L^{2}(\Omega_{T})},
\end{equation*}
for some constant $C_0>0$ independent of $v_F$ and $F$.
Choosing $E=C_0\|F\|_{L^{2}(\Omega_{T})}$, the condition \eqref{eq:apriori2}
satisfies, and then we can employ Corollary~\ref{cor:small2} with
$v=v_{F}$ to obtain 
\begin{equation}\label{eq:stability5}
	\|v_{F}\|_{L^{2}(0,T;{H}^{s-\gamma}(\Omega))}\le C\|F\|_{L^{2}(\Omega_{T})}\omega\LC \frac{\left\|(-\Delta)^{s}v_{F} \right\|_{L^{2}(0,T;H^{-s}(W))}}{\|F\|_{L^{2}(\Omega_{T})}}\RC .
\end{equation}

Meanwhile, for any function $u\in H^{-2}(0,T;H^{s-\gamma}(\Omega))$, by the duality argument, one has 
\begin{equation}
\|u\|_{H^{-2}(0,T;H^{s-\gamma}(\Omega))}\le\|u\|_{L^{2}(0,T;H^{s-\gamma}(\Omega))}.\label{eq:dual-space-rel1}
\end{equation}
Likewise, if $\bm{u}(T)=\partial_t\bm{u}(T)=0$, we can see that
\begin{equation}\label{eq:dual-space-rel2}
\left\|\partial_{t}^{2}u\right\|_{H^{-2}(0,T;H^{-s-\gamma}(\Omega))}\le\|u\|_{L^{2}(0,T;H^{-s-\gamma}(\Omega))}.
\end{equation}
Now, back to the backward wave equation \eqref{eq:backward-heat-equation}, if $\|q\|_{L^{\infty}(\Omega)}\le M$, by \eqref{eq:dual-space-rel1},
we have 
\begin{align}\label{eq:norm1}
	\begin{split}
		& \|F\|_{H^{-2}(0,T;H^{-s-\gamma}(\Omega))}\\
		\le&\left\|\partial_{t}^{2}v_{F}\right\|_{H^{-2}(0,T;H^{-s-\gamma}(\Omega))} \\
		&+\left\|(-\Delta)^{s}v_{F}\right\|_{H^{-2}(0,T;H^{-s-\gamma}(\Omega))}+\left\|qv_{F}\right\|_{H^{-2}(0,T;H^{-s-\gamma}(\Omega))} \\
		\le&\left\|\partial_{t}^{2}v_{F}\right\|_{H^{-2}(0,T;H^{-s-\gamma}(\Omega))}\\
		&+\left\|v_{F}\right\|_{H^{-2}(0,T;H^{s-\gamma}(\Omega))}+\|qv_{F}\|_{L^{2}(\Omega_{T})}\\
		\le& \left\|\partial_{t}^{2}v_{F}\right\|_{H^{-2}(0,T;H^{-s-\gamma}(\Omega))}+\|v_{F}\|_{L^{2}(0,T;H^{s-\gamma}(\Omega))}+M\|v_{F}\|_{L^{2}(\Omega_{T})}.
	\end{split}
\end{align}
Since $v_F$ satisfies $v_{F}=\partial_{t}v_{F}=0$ in $\mathbb{R}^{n}\times\{T\}$, via \eqref{eq:dual-space-rel2}, one obtains  
\begin{align}\label{eq:dual-space-rel3}
	\left\|\partial_{t}^{2}v_{F}\right\|_{H^{-2}(0,T;H^{-s-\gamma}(\Omega))}\le\|v_{F}\|_{L^{2}(0,T;H^{-s-\gamma}(\Omega))}.
\end{align}
Therefore, plugging \eqref{eq:dual-space-rel3} into \eqref{eq:norm1} yields 
\begin{equation}
	\|F\|_{H^{-2}(0,T;H^{-s-\gamma}(\Omega))}\le C\|v_{F}\|_{L^{2}(0,T;H^{s-\gamma}(\Omega))}\label{eq:norm2}
\end{equation}
provided $0<\gamma<s$. Combining \eqref{eq:stability5} and \eqref{eq:norm2}gives
\begin{equation}\label{eq:UCP1}
	\|F\|_{H^{-2}(0,T;H^{-s-\gamma}(\Omega))}\le C\|F\|_{L^{2}(\Omega_{T})}\omega\LC\frac{\left\|(-\Delta)^{s}v_{F}\right\|_{L^{2}(0,T;H^{-s}(W))}}{\|F\|_{L^{2}(\Omega_{T})}}\RC.
\end{equation}

We now investigate the Poisson operator $\mathcal{P}_{q}$ given in
\eqref{Poisson operator} in the following lemma. 
\begin{lem}\label{lem:Poisson-basic-properties}
 Suppose that $q\in L^{\infty}(\Omega)$.
	Let $\mathcal{P}_{q}$ be the Poisson operator given in \eqref{Poisson operator}.
	Then 
	\begin{equation}\label{eq:poisson-hilbert}
		\mathcal{P}_{q}-{\rm Id}:L^{2}(0,T;H_{\overline{W}}^{2s})\rightarrow L^{2}(\Omega_T)
	\end{equation}
	is a compact injective linear operator. Moreover, for each $F\in L^{2}(\Omega_{T})$,
	the adjoint operator of $\mathcal{P}_{q}-{\rm Id}$ is given by 
	\begin{equation}
		\LC \mathcal{P}_{q}-{\rm Id}\RC^{\ast}F=-(-\Delta)^{s}v_{F},\label{eq:adjoint}
	\end{equation}
	where $v_{F}$ is the solution of \eqref{eq:backward-heat-equation}. 
\end{lem}

\begin{proof}
	[Proof of Lemma~{\rm \ref{lem:Poisson-basic-properties}}] It is worth pointing out that the function $(\mathcal{P}_{q}-{\rm Id})f$ is equal to $\left. \mathcal{P}_q f \right|_{\Omega_T}$, for any $f\in L^{2}(0,T;H_{\overline{W}}^{2s})$. We split the proof into several steps.
	
	\vspace{3mm} 
	\noindent{\it Step 1. Compactness.}
	\vspace{3mm} 
	
	Using \eqref{eq:energy-est}, we can see that 
	\[
	\left\|(\mathcal{P}_{q}-{\rm Id})f\right\|_{L^{2}(0,T;\wt{H}^{s}(\Omega))}+\left\|\partial_t((\mathcal{P}_{q}-{\rm Id})f)\right\|_{L^{2}(0,T;L^2(\Omega))}\le C_{T} \|f\|_{L^{2}(0,T;H_{\overline{W}}^{2s})}.
	\]
	In other words, $\mathcal{P}_{q}-{\rm Id}:L^{2}(0,T;H_{\overline{W}}^{2s})\rightarrow L^{2}(0,T;\wt{H}^{s}(\Omega))\cap H^1(0,T;L^2(\Omega))$
	is a bounded linear operator. Since the embedding $\wt{H}^{s}(\Omega) \hookrightarrow L^{2}(\Omega)$ is compact, using the Aubin-Lions-Simon Theorem in \cite[Theorem II.5.16]{BF13tool}, we know that the embedding 
	\[
	L^{2}(0,T;\wt{H}^{s}(\Omega))\cap H^1(0,T;L^2(\Omega))\hookrightarrow L^{2}(\Omega_T) \quad \text{is compact}.
	\]
	Therefore, we see that the operator \eqref{eq:poisson-hilbert} is compact. 
	
	\vspace{3mm}
	\noindent{\it Step 2. Injectivity.}
	\vspace{3mm} 
	
	Suppose that $f\in\ker(\mathcal{P}_{q}-{\rm Id})$, then $\mathcal{P}_{q}f=0$
	in $\Omega_{T}$. From the definition of the Poisson operator, $u=\mathcal{P}_{q}f$ satisfies 
	\begin{equation}
		\LC \partial_{t}^{2}+(-\Delta)^{s}+q\RC  u=0\quad\text{ in }\;\,\Omega_{T}.\label{eq:Poisson-basic-properties1}
	\end{equation}
	Since $u=0$ in $\Omega_{T}$, from \eqref{eq:Poisson-basic-properties1}, we have that
	$(-\Delta)^{s}u=0$ in $\Omega_{T}$. Therefore, using the strong
	uniqueness for the fractional Laplacian again,
	we know that $u\equiv0$ throughout $\mathbb{R}^{n}$, and hence $f\equiv0$,
	which concludes that $\mathcal{P}_{q}$ is injective. 
	
	\vspace{3mm}
	\noindent{\it Step 3. Computing the adjoint operator.}
	\vspace{3mm}
	
	Given any $f\in L^{2}(0,T;H_{\overline{W}}^{2s})$, $F\in L^2(\Omega_T)$, we have 
	\begin{align*}
		\begin{split}
			& \LC(\mathcal{P}_{q}-{\rm Id})f,F \RC_{L^{2}(\Omega_{T})}\\
			=&\LC \mathcal{P}_{q}f,F\RC_{L^{2}(\Omega_{T})}\\
			=&\int_{\Omega_{T}} \mathcal{P}_{q}f \LC \partial_{t}^{2}v_{F}+ (-\Delta)^{s}v_{F}+ qv_{F}\RC dxdt\\
		    =&\int_{\Omega_{T}}\LC \partial_{t}^{2}\mathcal{P}_{q}f+q\mathcal{P}_{q}f\RC v_{F}\,dxdt+\LC\int_{(\mathbb{R}^{n})_{T}}-\int_{(\Omega_{e})_{T}}\RC(\mathcal{P}_{q}f)(-\Delta)^{s}v_{F}\,dxdt\\
			=&\int_{\Omega_{T}}\LC \partial_{t}^{2}\mathcal{P}_{q}f+(-\Delta)^{s}\mathcal{P}_{q}f+q\mathcal{P}_{q}f\RC v_{F}\,dxdt-\int_{(\Omega_{e})_{T}}(\mathcal{P}_{q}f)(-\Delta)^{s}v_{F}\,dxdt\\
			=&-\int_{(\mathbb{R}^{n})_{T}}f(-\Delta)^{s}v_{F}\,dxdt.
		\end{split}
	\end{align*}
	Consequently, the arbitrariness of $f$ implies \eqref{eq:adjoint}. 
\end{proof}
Combining \eqref{eq:UCP1} and \eqref{eq:adjoint}, we obtain the
following theorem. 
\begin{thm}
	\label{thm:UCP-quantify} Let $0<s<1$, $q\in L^{\infty}(\Omega)$
	with $\|q\|_{L^{\infty}(\Omega)}\le M$, and $W\subset\Omega_{e}$ be open
	such that $\overline{W}\cap\overline{\Omega}=\emptyset$. There exist
	a constant $C>1$ independent of $F$ and a logarithmic function $\omega$, both depending only on $n,s,\Omega,W,\gamma,T,M$,
	such that 
	\begin{equation} \label{eq:UCP2}
		\|F\|_{H^{-2}(0,T;H^{-s-\gamma}(\Omega))}\le C\|F\|_{L^{2}(\Omega_{T})}\omega\LC \frac{\left\|(\mathcal{P}_{q}-{\rm Id})^{*}F\right\|_{L^{2}(0,T;H^{-s}(W))}}{\|F\|_{L^{2}(\Omega_{T})}}\RC
	\end{equation}
	for all $F\in L^{2}(\Omega_{T})$. 
\end{thm}

\subsection{Quantitative Runge approximation}

Observe that $\LC \mathcal{P}_{q}-{\rm Id}\RC^{\ast}\LC \mathcal{P}_{q}-{\rm Id}\RC
$ is a compact, self-adjoint, positive definite operator on $L^{2}(0,T;H_{\overline{W}}^{2s})$ by Lemma~\ref{lem:Poisson-basic-properties}.
Therefore, there exist eigenvalues $\left\{ \mu_{j}\right\}_{j=1}^{\infty}$
with $\mu_{1}\ge\mu_{2}\ge\cdots\rightarrow0$ and eigenfunctions
$\left\{\varphi_{j}\right\}_{j=1}^{\infty}$ of $\LC \mathcal{P}_{q}-{\rm Id}\RC^{\ast}\LC \mathcal{P}_{q}-{\rm Id}\RC$.
Note that $\left\{\varphi_{j}\right\}_{j=1}^{\infty}$
forms an orthonormal basis of $L^{2}(0,T;H_{\overline{W}}^{2s})$.
Define 
\begin{align}\label{w_j}
	\sigma_{j}:=\sqrt{\mu_{j}}\quad\text{and}\quad w_{j}:=\frac{1}{\sigma_{j}}\LC\mathcal{P}_{q}-{\rm Id}\RC\varphi_{j}.
\end{align}
We can easily verify that 
\begin{align*}
	 \LC w_{j},w_{k}\RC_{L^{2}(\Omega_{T})} & =\frac{1}{\sigma_{j}\sigma_{k}}\LC (\mathcal{P}_{q}-{\rm Id})\varphi_{j},(\mathcal{P}_{q}-{\rm Id})\varphi_{k}\RC_{L^{2}(\Omega_{T})}\\
	& =\frac{1}{\sigma_{j}\sigma_{k}}\LC(\mathcal{P}_{q}-{\rm Id})^{\ast}(\mathcal{P}_{q}-{\rm Id})\varphi_{j},\varphi_{k}\RC_{L^{2}(0,T;H_{\overline{W}}^{2s})}\\
	& =\frac{\sigma_{j}}{\sigma_{k}}\LC\varphi_{j},\varphi_{k}\RC_{L^{2}(0,T;H_{\overline{W}}^{2s})}\\
	&=\delta_{jk},
\end{align*}
that is, $\left\{ w_{j}\right\}_{j=1}^{\infty}$ is an
orthonormal set in $L^{2}(\Omega_{T})$. Also, we have 
\begin{equation}
	\LC\mathcal{P}_{q}-{\rm Id}\RC^{\ast}w_{j}=\frac{1}{\sigma_{j}}\LC \mathcal{P}_{q}-{\rm Id}\RC ^{\ast}\LC \mathcal{P}_{q}-{\rm Id}\RC\varphi_{j}=\frac{\mu_{j}}{\sigma_{j}}\varphi_{j}=\sigma_{j}\varphi_{j}.\label{eq:adjoint-observ}
\end{equation}

\begin{lem}
		The set $\left\{w_{j} \right\}_{j=1}^{\infty}$ is complete in
	$L^{2}(\Omega_{T})$. In other words, it is an orthonormal basis of
	$L^{2}(\Omega_{T})$. 
\end{lem}

\begin{proof}
	Let $v\in L^{2}(\Omega_{T})$ be such that $\LC  v,w_{j}\RC_{L^{2}(\Omega_{T})}=0$
	for all $j$, then 
	\[
	\LC v,\mathcal{P}_{q}\varphi_{j}\RC_{L^{2}(\Omega_{T})}=\LC v,(\mathcal{P}_{q}-{\rm Id})\varphi_{j}\RC_{L^{2}(\Omega_{T})}=0,\quad\text{for all }j.
	\]
	Since $\left\{\varphi_{j}\right\}_{j=1}^{\infty}$ forms
	an orthonormal basis of $L^{2}(0,T;H_{\overline{W}}^{2s})$, then 
	\[
	\LC v,\mathcal{P}_{q}f \RC _{L^{2}(\Omega_{T})}=0\quad\text{ for all }f\in C_{c}^{\infty}(W_{T}).
	\]
	In view of the Runge approximation (Theorem~\ref{thm:Runge}), we conclude
	that $v\equiv 0$. 
\end{proof}
We now fix a number $\alpha>0$. We define the operator $R_{\alpha}:L^{2}(\Omega_{T})\rightarrow L^{2}(0,T;H_{\overline{W}}^{2s})$
by the finite sum 
\begin{equation*}
	R_{\alpha}\phi:=\sum_{\sigma_{j}>\alpha}\frac{1}{\sigma_{j}}\LC \phi,w_{j}\RC_{L^{2}(\Omega_{T})}\varphi_{j}.
\end{equation*}
Since $\left\{\varphi_{j}\right\}_{j=1}^{\infty}$
is an orthonormal basis of $L^{2}(0,T;H_{\overline{W}}^{2s})$, and $\left\{\sigma_{j} \right\}$
is non-increasing, using Parseval's identity, then it is easy to see that 
\begin{equation}\label{eq:SVD1}
	\begin{split}
		\left\|R_{\alpha}\phi\right\|_{L^{2}(0,T;H_{\overline{W}}^{2s})}^{2}=&\sum_{\sigma_{j}>\alpha}\frac{1}{\sigma_{j}^{2}}|(\phi,w_{j})_{L^{2}(\Omega_{T})}|^{2}\\
		\le & \frac{1}{\alpha^{2}}\|\phi\|_{L^{2}(\Omega_{T})}^{2},
	\end{split}
\end{equation}
and
\begin{align}\label{eq:SVD2}
	\begin{split}
		\left\|(\mathcal{P}_{q}-{\rm Id})(R_{\alpha}\phi)-\phi\right\|_{L^{2}(\Omega_{T})}^{2}=&\left\|\sum_{\sigma_{j}>\alpha}\frac{1}{\sigma_{j}}(\phi,w_{j})_{L^{2}(\Omega_{T})}(\mathcal{P}_{q}-{\rm Id})\varphi_{j}-\phi\right\|_{L^{2}(\Omega_{T})}^{2} \\
		=&\left\|\sum_{\sigma_{j}>\alpha}(\phi,w_{j})_{L^{2}(\Omega_{T})}w_{j}-\phi\right\|_{L^{2}(\Omega_{T})}^{2}\\
		=&\left\|\sum_{\sigma_{j}\le\alpha}(\phi,w_{j})_{L^{2}(\Omega_{T})}w_{j}\right\|_{L^{2}(\Omega_{T})}^{2} \\
		 =&\sum_{\sigma_{j}\le\alpha}\left|(\phi,w_{j})_{L^{2}(\Omega_{T})}\right|^{2},
	\end{split}
\end{align}
where we have used the orthonormality of $\left\{ w_j\right\}_{j=1}^\infty$ in $L^2(\Omega_T)$.

Let us define 
\begin{equation}
	r_{\alpha}:=\sum_{\sigma_{j}\le\alpha}\LC \phi,w_{j}\RC_{L^{2}(\Omega_{T})}w_{j}.\label{eq:SVD-expansion2}
\end{equation}
In particular, for any $\phi \in H_0^{2}(0,T; \wt H^{s+\gamma}(\Omega))\subset L^2(\Omega_T)$, combining \eqref{eq:SVD2} and \eqref{eq:SVD-expansion2}, we have
\begin{align}
	\begin{split}
		& \left\|\LC\mathcal{P}_{q}-{\rm Id}\RC \LC R_{\alpha}\phi\RC -\phi\right\|_{L^{2}(\Omega_{T})}^{2}\\
		=&\left|\LC \phi,r_{\alpha}\RC_{L^{2}(\Omega_{T})}\right| \\
		 \le &\|\phi\|_{H_0^{2}(0,T; \wt H^{s+\gamma}(\Omega))}\left\|r_{\alpha}\right\|_{H^{-2}(0,T;H^{-s-\gamma}(\Omega))}.\label{eq:SVD3}
	\end{split}
\end{align}
We now choose $F=r_{\alpha}\in L^{2}(\Omega_{T})$ in \eqref{eq:UCP2},
and we obtain 
\[
\left\|r_{\alpha}\right\|_{H^{-2}(0,T;H^{-s-\gamma}(\Omega))}\le C\left\|r_{\alpha}\right\|_{L^{2}(\Omega_{T})}\omega\LC\frac{\|(\mathcal{P}_{q}-{\rm Id})^{\ast}r_{\alpha}\|_{L^{2}(0,T;H^{-s}(W))}}{\|r_{\alpha}\|_{L^{2}(\Omega_{T})}}\RC,
\]
and thus \eqref{eq:SVD3} implies 
\begin{equation}
	\begin{split}
		&\left\|(\mathcal{P}_{q}-{\rm Id})(R_{\alpha}\phi)-\phi\right\|_{L^{2}(\Omega_{T})}^{2}\\
		\le & C\|\phi\|_{H_0^{2}(0,T; \wt H^{s+\gamma}(\Omega))}\|r_{\alpha}\|_{L^{2}(\Omega_{T})}\omega\LC\frac{\|(\mathcal{P}_{q}-{\rm Id})^{\ast}r_{\alpha}\|_{L^{2}(0,T;H^{-s}(W))}}{\|r_{\alpha}\|_{L^{2}(\Omega_{T})}}\RC.\label{eq:SVD4}
	\end{split}
\end{equation}
In view of \eqref{eq:adjoint-observ}, we have 
\[
\LC \mathcal{P}_{q}-{\rm Id}\RC^{\ast}r_{\alpha}=\sum_{\sigma_{j}\le\alpha}\LC \phi,w_{j}\RC_{L^{2}(\Omega_{T})}(\mathcal{P}_{q}-{\rm Id})^{\ast}w_{j}=\sum_{\sigma_{j}\le\alpha}(\phi,w_{j})_{L^{2}(\Omega_{T})}\sigma_{j}\varphi_{j}.
\]
From the property that $\left\{\varphi_{j}\right\}_{j=1}^{\infty}$ is an orthonormal basis of $L^{2}(0,T;H_{\overline{W}}^{2s})$, it follows 
\begin{align*}
	&\left\|(\mathcal{P}_{q}-{\rm Id})^{\ast}r_{\alpha}\right\|_{L^{2}(0,T;H^{-s}(W))}^{2} \\
	\le &\|(\mathcal{P}_{q}-{\rm Id})^{*}r_{\alpha}\|_{L^{2}(0,T;H_{\overline{W}}^{2s})}^{2}\\
	= &\sum_{\sigma_{j}\le\alpha} \sigma_{j}^{2} \left|(\phi,w_{j})_{L^{2}(\Omega_{T})}\right|^{2}\\
	\le & \alpha^{2}\sum_{\sigma_{j}\le\alpha}\left|(\phi,w_{j})_{L^{2}(\Omega_{T})}\right|^{2}\\
	=&\alpha^{2}\|r_{\alpha}\|_{L^{2}(\Omega_{T})}^{2} ,
\end{align*}
where we have used \eqref{eq:SVD-expansion2} in the last equality.
Since $\omega$ is monotone non-decreasing, \eqref{eq:SVD4} gives 
\begin{equation}
	\left\|(\mathcal{P}_{q}-{\rm Id})(R_{\alpha}\phi)-\phi\right\|_{L^{2}(\Omega_{T})}^{2}\le C\|\phi\|_{H_0^{2}(0,T;\wt H^{s+\gamma}(\Omega))}\left\|r_{\alpha}\right\|_{L^{2}(\Omega_{T})}\omega(\alpha).\label{eq:SVD5}
\end{equation}
Furthermore, observe that 
\begin{align}
	\begin{split}
		& (\mathcal{P}_{q}-{\rm Id})(R_{\alpha}\phi)-\phi \\
		= & \sum_{\sigma_{j}>\alpha}\frac{1}{\sigma_{j}}(\phi,w_{j})_{L^{2}(\Omega_{T})}(\mathcal{P}_{q}-{\rm Id})\varphi_{j}-\phi \\
		=&\sum_{\sigma_{j}>\alpha}(\phi,w_{j})_{L^{2}(\Omega_{T})}w_{j}-\phi \\
		= & -\sum_{\sigma_{j}\le\alpha}(\phi,w_{j})_{L^{2}(\Omega_{T})}w_{j}\\
		=&-r_{\alpha}.\label{eq:SVD6}
	\end{split}
\end{align}
Combining \eqref{eq:SVD5} and \eqref{eq:SVD6} yields 
\begin{equation}
	\left\|(\mathcal{P}_{q}-{\rm Id})(R_{\alpha}\phi)-\phi\right\|_{L^{2}(\Omega_{T})}\le C\|\phi\|_{H^{2}_0(0,T;\wt H^{s+\gamma}(\Omega))}\omega(\alpha).\label{eq:SVD7}
\end{equation}

Given any $\epsilon>0$, there exists a unique $\alpha>0$ such that
$\epsilon=\omega(\alpha)$. Write $f_{\epsilon}=R_{\alpha}\phi$,
and we know that \eqref{eq:SVD1} can be rewritten as 
\begin{equation}\label{eq:SVD8}
 \left\|f_{\epsilon}\right\|_{L^{2}(0,T;H_{\overline{W}}^{2s})}\le\frac{1}{\omega^{-1}(\epsilon)}\|\phi\|_{L^{2}(\Omega_{T})},
\end{equation}
where $\omega^{-1}$ is the inverse function of $\omega$.
Now, as in \cite[Remark 3.4]{RS20Calderon}, since 
\[
\omega(t)=C|\log t|^{-\sigma}\;\; \text{ for }t \text{ small},
\]
we can take $\frac{1}{\omega^{-1}(\eps)}\le\exp (\wt C\eps^{-\mu})$ with $\wt C\ge C^{1/\sigma}$ and $\mu =1/\sigma$ for all $\eps>0$.
Therefore, restating \eqref{eq:SVD7} and \eqref{eq:SVD8} leads to the following
theorem. 
\begin{thm}[Quantitative Runge approximation] \label{thm:Quantative-Runge}
 Let $\|q\|_{L^{\infty}(\Omega)}\le M$ and fix a parameter $\gamma>0$. Given
	any $\phi\in H^{2}_0(0,T;\wt H^{s+\gamma}(\Omega))$, and any $\epsilon>0$,
	there exists $f_{\epsilon}\in L^{2}(0,T;H_{\overline{W}}^{2s})$
	such that 
	\begin{subequations}
		\begin{equation*}
			\left\|\mathcal{P}_{q}f_{\epsilon}-f_{\epsilon}-\phi \right\|_{L^{2}(\Omega_{T})}\le C\|\phi\|_{H^{2}_0(0,T;\wt H^{s+\gamma}(\Omega))}\epsilon,
		\end{equation*}
		and 
		\begin{equation*}
			\left\|f_{\epsilon}\right\|_{L^{2}(0,T;H_{\overline{W}}^{2s})}\le C\exp(\wt{C}\epsilon^{-\mu})\|\phi\|_{L^{2}(\Omega_{T})},
		\end{equation*}
	\end{subequations}
	for some positive constants $\mu$, $C$ and $\wt{C}$, depending
	only on $n,s,\gamma,\Omega,W,\gamma,T,M$. 
\end{thm}

\subsection{Proof of Theorem~\ref{thm:main-stability}}

Finally, we can prove our logarithmic stability estimate of the inverse problem for the fractional wave equation.
\begin{proof}
	[Proof of Theorem~{\rm \ref{thm:main-stability}}] Let $\epsilon>0$
	be a parameter to be chosen later. We fix two arbitrary functions
	$\phi_{j}\in H^{2}_0(0,T;\wt H^{s+\gamma}(\Omega))$ with $\|\phi_{j}\|_{H^{2}_0(0,T;\wt H^{s+\gamma}(\Omega))}=1$, for $j=1,2$.
	Using Theorem~\ref{thm:Quantative-Runge}, there exist functions $f_{j}\in L^{2}(0,T;H_{\overline{W}}^{2s})$
	such that 
	\[
	\|t_{j}\|_{L^{2}(\Omega_{T})}\le C\epsilon\quad\text{and}\quad\|f_{j}\|_{L^{2}(0,T;H_{\overline{W}}^{2s})}\le C\exp(\wt{C}\epsilon^{-\mu})
	\]
	with 
	\[
	t_{j}=u_{j}-f_{j}-\phi_{j}\quad\text{and}\quad u_{j}=\mathcal{P}_{q_{j}}f_{j},
	\]
	where we have used the fact 
	\[
	\|\phi_{j}\|_{L^{2}(\Omega_{T})}\le\|\phi_{j}\|_{H^{2}_0(0,T;\wt H^{s+\gamma}(\Omega))}=1.
	\]
	Inserting $u_{j}$ into the identity \eqref{eq:crucial-identity}
	in Lemma~\ref{lem:self-adjoint}, we obtain 
	\begin{align*}
		& \int_{\Omega_{T}}(q_{1}-q_{2})\phi_{1}\phi^\ast _{2}\, dxdt\\
		 =&\int_{(\Omega_{e})_{T}}\LC (\Lambda_{q_{1}}-\Lambda_{q_{2}})f_{1}\RC  f^\ast_{2}\, dxdt-\int_{\Omega_{T}}(q_{1}-q_{2})\LC \phi_{2}t^\ast_{1}+\phi_{1}t^\ast_{2}+t_{1}t^\ast_{2}\RC dxdt.
	\end{align*}
	Therefore, 
	\begin{align*}
		& \left|\int_{\Omega_{T}}(q_{1}-q_{2})\phi_{1}\phi ^\ast_{2}\, dxdt\right|\\
		 \le& \left\|\Lambda_{q_{1}}-\Lambda_{q_{2}}\right\|_{\ast}\|f_{1}\|_{L^{2}(0,T;H_{\overline{W}}^{2s})}\|f_{2}\|_{L^{2}(0,T;H_{\overline{W}}^{2s})}\\
		&+2M\LC\|t_{1}\|_{L^{2}(\Omega_{T})}+\|t_{2}\|_{L^{2}(\Omega_{T})}+\|t_{1}\|_{L^{2}(\Omega_{T})}\|t_{2}\|_{L^{2}(\Omega_{T})}\RC\\
		 \le& C^{2}\left\|\Lambda_{q_{1}}-\Lambda_{q_{2}}\right\|_{\ast}\exp\LC 2\wt{C}\epsilon^{-\mu}\RC +4CM\epsilon.
	\end{align*}
	Choosing $\epsilon=\left |\log\LC \|\Lambda_{q_{1}}-\Lambda_{q_{2}}\|_{\ast}\RC\right|^{-\frac{1}{\mu}}$,
	we know that 
	\[
	\left|\int_{\Omega_{T}}(q_{1}-q_{2})\phi_{1}\phi^\ast_{2}\, dxdt\right|\le\omega\LC\left\|\Lambda_{q_{1}}-\Lambda_{q_{2}}\right\|_{\ast}\RC,
	\]
	for some logarithmic modulus of continuity (which is monotone non-decreasing)
	$\omega$. Recalling the definition of the function space $Z^{-s}(\Omega,T)$ in Definition \ref{defi: space Z},  we finally prove the assertion.
\end{proof}

\section{Exponential instability of the inverse problem \label{sec:optimality}}

In the last section of this paper, we demonstrate that the logarithmic stability in Theorem \ref{thm:main-stability} is optimal, by showing the exponential instability phenomenon for the fractional wave equation. The ideas of the construction of the instability are motivated by Mandache's pioneer work \cite{Man01instability}.

\subsection{Matrix representation via an orthonormal basis}
For $r>0$, let $B_r$ be the ball of radius $r>0$ with center at $0$.
First of all, we introduce a set of basis of $L^{2}(B_{3}\setminus\overline{B_{2}})$.
The following proposition can be found in \cite[Lemma 2.1 and Remark 2.2]{RS18Instability}: 
\begin{prop}
	\label{prop:opt-special-basis} Let $n \ge 2$. Given any $m\ge0$, we define 
	\[
	\ell_{m}:=\begin{pmatrix}m+n-1\\
		n-1
	\end{pmatrix}-\begin{pmatrix}m+n-3\\
		n-1
	\end{pmatrix}\le2(1+m)^{n-2}.
	\]
	There exists an orthonormal basis $\left\{ Y_{mk\ell}:\,  m\ge0,k\ge0,0\le\ell\le\ell_{m} \right\}$
	of $L^{2}(B_{3}\setminus\overline{B_{2}})$ such that 
	\[
	\left\|\tilde{Y}_{mk\ell}\right\|_{L^{2}(B_{1})}\le C_{n,s}'e^{-C_{n,s}''(m+k)}
	\]
	for some constant $C_{n,s}'$ and $C_{n,s}''$ (both depending only
	on $n$ and $s$), where $\tilde{Y}_{mk\ell}\in H^{s}(\mathbb{R}^{n})$
	is the unique solution to 
	\[
	\begin{cases}
		(-\Delta)^{s}\tilde{Y}_{mk\ell}=0 & \text{ in }B_{1},\\
		\tilde{Y}_{mk\ell}=\mathbbm1_{B_{3}\setminus\overline{B_{2}}}Y_{mk\ell} & \text{ in }\mathbb{R}^{n}\setminus\overline{B_{1}}.
	\end{cases}
	\]
\end{prop}

\begin{rmk}
For $n=1$, the "sphere" $\partial B_{1}\subset \R$ consists only two end points $\{-1,1\}$, which is no longer a sphere. Therefore, we need to find another basis for the one-dimensional case. We shall discuss the case of $n=1$ later. 
\end{rmk}

Given any bounded linear operator $\mathcal{A}:L^{2}(B_{3}\setminus\overline{B_{2}})\rightarrow L^{2}(B_{3}\setminus\overline{B_{2}})$,
we define 
\[
a_{m_{1}k_{1}\ell_{1}}^{m_{2}k_{2}\ell_{2}}:=\LC \mathcal{A}Y_{m_{1}k_{1}\ell_{1}},Y_{m_{2}k_{2}\ell_{2}}\RC_{L^{2}(B_{3}\setminus\overline{B_{2}})}.
\]
Let $\LC a_{m_{1}k_{1}\ell_{1}}^{m_{2}k_{2}\ell_{2}}\RC$ be the tensor
with entries $a_{m_{1}k_{1}\ell_{1}}^{m_{2}k_{2}\ell_{2}}$, and consider
the following Banach space: 

\[
X:=\left\{ \LC a_{m_{1}k_{1}\ell_{1}}^{m_{2}k_{2}\ell_{2}}\RC :\, \left\|\LC a_{m_{1}k_{1}\ell_{1}}^{m_{2}k_{2}\ell_{2}}\RC \right\|_{X}<\infty \right\},
\]
where 
\begin{equation}
	\left\|\LC  a_{m_{1}k_{1}\ell_{1}}^{m_{2}k_{2}\ell_{2}}\RC \right\|_{X}:=\sup_{m_{i}k_{i}\ell_{i}}\LC 1+\max\{m_{1}+k_{1},m_{2}+k_{2}\}\RC^{n+2}\left|A_{m_{1}k_{1}\ell_{1}}^{m_{2}k_{2}\ell_{2}}\right
	|,\label{eq:opt-X-norm}
\end{equation}
see e.g. \cite[Definition 2.7]{RS18Instability}. The following lemma
can be found in \cite[(20)]{RS18Instability}, which plays an essential role in our work. 
\begin{lem}\label{lem:opt-matrix-repn}
	If $n \ge 2$, then
	\begin{equation*}
		\|\mathcal{A}\|_{L^{2}(B_{3}\setminus\overline{B_{2}})\rightarrow L^{2}(B_{3}\setminus\overline{B_{2}})}\le4\left\|\LC  a_{m_{1}k_{1}\ell_{1}}^{m_{2}k_{2}\ell_{2}}\RC \right\|_{X}
	\end{equation*}
\end{lem}

Thanks to Lemma~\ref{lem:opt-matrix-repn}, we can regard the
tensor $\LC  a_{m_{1}k_{1}\ell_{1}}^{m_{2}k_{2}\ell_{2}}(q)\RC $ as the
\emph{matrix representation} of the bounded linear operator $\mathcal{A}$. 

\subsection{Special weak solutions}

In view of Proposition~\ref{prop:opt-special-basis}, we need to introduce some special solutions. We begin with the following lemma. 
\begin{lem}
\label{lem:patch2-veryweak-wellposed}Let $\chi=\chi(t)\in C_{c}^{\infty}((0,T))$
and 
\[
q\in B_{+,R}^{\infty}:=\left\{ q\text{ is real-valued}:\,  0\le q\le R\text{ a.e.}\right\} .
\]
Given any $f=f(x)\in L^{2}(B_{3}\setminus\overline{B_{2}})$, there
exists a unique solution $u$ to 
\begin{equation}
\begin{cases}
\LC \partial_{t}^{2}+(-\Delta)^{s}+q\RC u=0 & \text{in}\;\;(B_{1})_{T},\\
u(x,t)=\chi(t)\mathbbm1_{B_{3}\setminus\overline{B_{2}}}(x)f(x) & \text{in}\;\;(\mathbb{R}^{n}\setminus\overline{B_{1}})_{T},\\
u=\partial_{t}u=0 & \text{in}\;\;\mathbb{R}^{n}\times\{0\},
\end{cases}\label{eq:patch2-veryweak1a}
\end{equation}
and  
\begin{equation}
\|u\|_{L^{\infty}(0,T;L^{2}(B_{1}))}\le C_{R,T,n,s}\|\chi\|_{W^{2,\infty}(0,T)}\|f\|_{L^{2}(B_{3}\setminus\overline{B}_{2})}\label{eq:patch2-veryweak1b}
\end{equation}
for some positive constant $C_{R,T,n,s}$. 
\end{lem}

\begin{proof}
[Proof of Lemma~{\rm \ref{lem:patch2-veryweak-wellposed}}]Recall from \cite[Remark 2.2]{RS18Instability},
there exists a unique solution $\tilde{f}$ to 
\begin{equation}
\begin{cases}
(-\Delta)^{s}\tilde{f}=0 & \text{in}\,\,B_{1},\\
\tilde{f}=\mathbbm1_{B_{3}\setminus\overline{B_{2}}}f & \text{in}\,\,\mathbb{R}^{n}\setminus\overline{B_{1}},
\end{cases}\label{eq:patch2-tildef-def}
\end{equation}
such that 
\begin{equation}
\|\tilde{f}\|_{L^{2}(\mathbb{R}^{n})}\le C_{n,s}\|f\|_{L^{2}(B_{3}\setminus\overline{B}_{2})}\label{eq:patch2-veryweak2}
\end{equation}
for some constant $C_{n,s}>0$. Let $F(x,t):=-(\chi''(t)+\chi(t)q(x))\tilde{f}(x).$
Since $F\in L^{2}(B_{1}\times(0,T))$, using Theorem~\ref{thm:well-posedness},
there exists a unique solution $v$ to 
\begin{equation}
\begin{cases}
\LC \partial_{t}^{2}+(-\Delta)^{s}+q\RC v=F & \text{in}\;\;(B_{1})_{T},\\
v=0 & \text{in}\,\,(\mathbb{R}^{n}\setminus\overline{B_{1}})_{T},\\
v=\partial_{t}v=0 & \text{in}\;\;\mathbb{R}^{n}\times\{0\},
\end{cases}\label{eq:patch2-veryweak3a}
\end{equation}
satisfying 
\begin{equation}
\|v\|_{L^{\infty}(0,T;\tilde{H}^{s}(B_{1}))}+\|\partial_{t}v\|_{L^{\infty}(0,T;L^{2}(B_{1}))}\le C_{R,T,n,s}\|F\|_{L^{2}(B_{1}\times(0,T))}.\label{eq:patch2-veryweak3b}
\end{equation}
In other words, 
\begin{equation}
u(x,t):=v(x,t)+\chi(t)\tilde{f}(x)\label{eq:patch2-veryweak4}
\end{equation}
is the unique solution to \eqref{eq:patch2-veryweak1a}. Therefore,
we can obtain from \eqref{eq:patch2-veryweak3b} that 
\begin{align}\label{eq:patch2-veryweak5}
\begin{split}
	 \|u\|_{L^{\infty}(0,T;L^{2}(B_{1}))}&\le\|v\|_{L^{\infty}(0,T;L^{2}(B_{1}))}+\|\chi\|_{L^{\infty}(0,T)}\|\tilde{f}\|_{L^{2}(B_{1})} \\
	& \le C_{R,T,n,s}\|\chi\|_{W^{2,\infty}(0,T)}\|\tilde{f}\|_{L^{2}(B_{1})}.
\end{split}
\end{align}
Combining \eqref{eq:patch2-veryweak2} and \eqref{eq:patch2-veryweak5} implies \eqref{eq:patch2-veryweak1b}. 
\end{proof}
Based on Lemma~\ref{lem:patch2-veryweak-wellposed}, we can define the DN map
\begin{equation}
(\Lambda_{q}\chi)(f)=\Lambda_{q}(\chi f):=(-\Delta)^{s}u|_{(\mathbb{R}^{n}\setminus\overline{B_{1}})_{T}}\quad\text{for all}\,\,f\in L^{2}(B_{3}\setminus\overline{B_{2}}),\label{eq:patch2-weak-DN-map}
\end{equation}
where $u$ is given in \eqref{eq:patch2-veryweak1a}. In view of \eqref{eq:patch2-veryweak1b},
we know that 
\begin{equation}
\Lambda_{q}\chi:L^{2}(B_{3}\setminus\overline{B}_{2})\rightarrow L^{\infty}(0,T;H^{-2s}(B_{3}\setminus\overline{B_{2}}))\label{eq:patch2-DN-map-regularity1}
\end{equation}
is a bounded operator. However, the regularity given in \eqref{eq:patch2-DN-map-regularity1}
is insufficient for our purpose. In the following lemma, we improve
\eqref{eq:patch2-DN-map-regularity1} by modifying the ideas in \cite[Remark 2.5]{RS18Instability}. 
\begin{lem}
\label{lem:patch2-crucial-regularity-DN}The operator $\Lambda_{q}\chi:L^{2}(B_{3}\setminus\overline{B}_{2})\rightarrow L^{\infty}(0,T;L^{2}(B_{3}\setminus\overline{B_{2}}))$
is bounded. Precisely, the following estimate holds:
\begin{equation}
\|(\Lambda_{q}\chi)(f)\|_{L^{\infty}(0,T;L^{2}(B_{3}\setminus\overline{B_{2}}))}\le C_{R,T,n,s}\|\chi\|_{W^{2,\infty}(0,T)}\|f\|_{L^{2}(B_{3}\setminus\overline{B}_{2})}\label{eq:patch-improve-regularity1}
\end{equation}
for all $f\in L^{2}(B_{3}\setminus\overline{B}_{2})$. 
\end{lem}

\begin{proof}
From \eqref{eq:patch2-tildef-def}, \eqref{eq:patch2-veryweak4},
and \eqref{eq:patch2-weak-DN-map}, it follows that 
\[
(\Lambda_{q}\chi)(f)=(-\Delta)^{s}v|_{(\mathbb{R}^{n}\setminus\overline{B_{1}})_{T}}, \quad\text{for all}\,\,f\in L^{2}(B_{3}\setminus\overline{B_{2}}),
\]
where $v$ solves \eqref{eq:patch2-veryweak3a}.
Using the equivalent definition of $(-\Delta)^{s}$ via the singular integral,
see e.g. \cite[Definition 2.5, Theorem 5.3]{kwasnicki2017ten}, we obtain that for
each $x\in \R^n \setminus \overline{B_1}$
\[
|(-\Delta)^{s}v(x,t)|=C_{n,s}\left|\int_{\mathbb{R}^{n}}\frac{v(x,t)-v(y,t)}{|x-y|^{n+2s}}\,dy\right|=C_{n,s}\left|\int_{B_{1}}\frac{v(y,t)}{|x-y|^{n+2s}}\,dy\right|
\]
for some positive constant $C_{n,s}$. Since 
\[
|x-y|^{n+2s}\ge1, \quad\text{for all}\,\,x\in B_{3}\setminus\overline{B_{2}},y\in B_{1},
\]
we immediately observe that 
\[
|(-\Delta)^{s}v(x,t)|\le C_{n,s}\left|\int_{B_{1}}v(y,t)\,dy\right|\le C_{n,s}\|v(\bullet,t)\|_{L^{2}(B_{1})}.
\]
Hence, with the help of  \eqref{eq:patch2-veryweak5} and \eqref{eq:patch2-veryweak2}, we have
\begin{align*}
  &\|(\Lambda_{q}\chi)(f)\|_{L^{\infty}(0,T;L^{2}(B_{3}\setminus\overline{B_{2}}))} \\=&\|(-\Delta)^{s}v\|_{L^{\infty}(0,T;L^{2}(B_{3}\setminus\overline{B_{2}}))}\\
 \le& C_{n,s}\|v\|_{L^{\infty}(0,T;L^{2}(B_{1}))}\\
 \le & C_{R,T,n,s}\|\chi\|_{W^{2,\infty}(0,T)}\|\tilde{f}\|_{L^{2}(B_{1})}\\
 \le& C_{R,T,n,s}\|\chi\|_{W^{2,\infty}(0,T)}\|f\|_{L^{2}(B_{3}\setminus\overline{B}_{2})},
\end{align*}
which proves the lemma.
\end{proof}
We now apply Lemma~\ref{lem:patch2-veryweak-wellposed} with the exterior Dirichlet data $f=Y_{mk\ell}$. Since $\{Y_{mk\ell}\}$
is an orthonormal basis of $L^{2}(B_{3}\setminus\overline{B}_{2})$,
by Lemma~\ref{lem:patch2-veryweak-wellposed}, there exists a
unique solution $u_{mk\ell}$ to 
\begin{equation}
\begin{cases}
\LC \partial_{t}^{2}+(-\Delta)^{s}+q\RC  u_{mk\ell}=0 & \text{in}\;\;(B_{1})_{T},\\
u_{mk\ell}(x,t)=\chi(t)\mathbbm1_{B_{3}\setminus\overline{B_{2}}}(x)Y_{mk\ell}(x) & \text{in}\;\;(\mathbb{R}^{n}\setminus\overline{B_{1}})_{T},\\
u_{mk\ell}=\partial_{t}u_{mk\ell}=0 & \text{in}\;\;\mathbb{R}^{n}\times\{0\},
\end{cases}\label{eq:opt-well-posed-weak1}
\end{equation}
satisfying 
\begin{equation}
\|u_{mk\ell}\|_{L^{\infty}(0,T;L^{2}(B_{1}))}\le C_{T,R,n,s}\|\chi\|_{W^{2,\infty}(0,T)}.\label{eq:patch2-well-notsharp}
\end{equation}
Thanks to Proposition~\ref{prop:opt-special-basis}, \eqref{eq:patch2-well-notsharp} can be improved to  
\begin{equation}\label{lem:opt-basic-est}
\|u_{mk\ell}\|_{L^{\infty}(0,T;L^{2}(B_{1}))}\le C_{T,R,n,s}\|\chi\|_{W^{2,\infty}(0,T)}e^{-c_{n,s}(m+k)}
\end{equation}

\begin{rmk}
Similarly, we can let $\mathring{u}_{mk\ell}$ be the solution to
\eqref{eq:opt-well-posed-weak1} with respect to $q\equiv0$ (i.e. $R=0$), we
have 
\[
\|u_{mk\ell}\|_{L^{\infty}(0,T;L^{2}(B_{1}))}\le C_{T,n,s}\|\chi\|_{W^{2,\infty}(0,T)}e^{-c_{n,s}(m+k)}.
\]
\end{rmk}

\subsection{Matrix representation}

We consider the mapping 
\[
\Gamma(q)(f):=\chi\Lambda_{q}(\chi f)-\chi\Lambda_{0}(\chi f)\quad\text{ for }\;f\in C_{c}^{\infty}((\mathbb{R}^{n}\setminus\overline{B_{1}})_{T}),
\]
where $\Lambda_{0}$ is the DN map \eqref{eq:hyperbolic-DN} with
$q\equiv0$. We define
\begin{align}\label{continuous differences}
	\Gamma_{m_{1}k_{1}\ell_{1}}^{m_{2}k_{2}\ell_{2}}(q)(t):= \LC \Gamma(q)(Y_{m_{1}k_{1}\ell_{1}}),Y_{m_{2}k_{2}\ell_{2}}\RC _{L^{2}(B_{3}\setminus\overline{B_{2}})}(t).
\end{align}
Since $\Lambda_{q}$ is self-adjoint, then \eqref{continuous differences} infers that 
\begin{equation*}
	\Gamma_{m_{1}k_{1}\ell_{1}}^{m_{2}k_{2}\ell_{2}}(q)(t)=\Gamma_{m_{2}k_{2}\ell_{2}}^{m_{1}k_{1}\ell_{1}}(q)(t).
\end{equation*}

We have the following estimate for the tensor. 
\begin{lem}\label{lem:opt-matrix-repn-est}
	For $n \ge 2$ and given $q\in B_{+,R}^{\infty}$,
	there exist constants $C'_{R,T,n,s}>1$ and $c'_{n,s}>0$ such that
	\begin{equation}\label{eq:opt-matrix-repn-est3}
		\sup_{t\in(0,T)}\left|\Gamma_{m_{1}k_{1}\ell_{1}}^{m_{2}k_{2}\ell_{2}}(q)(t)\right|\le C'_{R,T,n,s}\|\chi\|_{W^{2,\infty}(0,T)}^{2} e^{-c'_{n,s}\sigma}
	\end{equation}
	where $	\sigma:=\max\left\{m_{1}+k_{1},m_{2}+k_{2}\right\}$.
\end{lem}

\begin{proof}
	Since $\left\|Y_{m_{2}k_{2}\ell_{2}}\right\|_{L^{2}(B_{3}\setminus\overline{B_{2}})}=1$,
	using the equivalent definition of $(-\Delta)^{s}$ again, we have 
	\begin{align*}
		& \quad  \sup_{t\in(0,T)}\left|\Gamma_{m_{1}k_{1}\ell_{1}}^{m_{2}k_{2}\ell_{2}}(q)(t)\right|^{2}\\
		& \le\sup_{t\in(0,T)}\left\|\Gamma(q)(Y_{m_{1}k_{1}\ell_{1}})\right\|_{L^{2}(B_{3}\setminus\overline{B_{2}})}^{2}(t)\\
		& \le\|\chi\|_{L^{\infty}(0,T)}^{2}\sup_{t\in(0,T)}\left\|(-\Delta)^{s}w_{m_{1}k_{1}\ell_{1}}\right\|_{L^{2}(B_{3}\setminus\overline{B_{2}})}^{2}(t) \\
		 &\qquad (w_{m_{1}k_{1}\ell_{1}}:=u_{m_{1}k_{1}\ell_{1}}-\mathring{u}_{m_{1}k_{1}\ell_{1}})\\
		& =\|\chi\|_{L^{\infty}(0,T)}^{2}C_{n,s}\sup_{t\in(0,T)}\int_{B_{3}\setminus\overline{B_{2}}}\left|\int_{\mathbb{R}^{d}}\frac{w_{m_{1}k_{1}\ell_{1}}(x,t)-w_{m_{1}k_{1}\ell_{1}}(y,t)}{|x-y|^{n+2s}}\,dy\right|^{2}\,dx\\
		&\le\|\chi\|_{L^{\infty}(0,T)}^{2}C_{n,s}\sup_{t\in(0,T)}\int_{B_{3}\setminus\overline{B_{2}}}\left|\int_{B_{1}}w_{m_{1}k_{1}\ell_{1}}(y,t)\,dy\right|^{2}\,dx\\
		&\qquad(\text{since}\;\;w_{m_{1}k_{1}\ell_{1}}=0\text{ in }(B_{3}\setminus\overline{B_{2}})_{T})\\
		& =\|\chi\|_{L^{\infty}(0,T)}^{2}C_{n,s}\sup_{t\in(0,T)}|B_{3}\setminus\overline{B_{2}}|\left|\int_{B_{1}}w_{m_{1}k_{1}\ell_{1}}(y,t)\,dy\right|^{2}\, dx\\
		& \le\|\chi\|_{L^{\infty}(0,T)}^{2}C_{n,s}\sup_{t\in(0,T)}|B_{3}\setminus\overline{B_{2}}||B_{1}|\int_{B_{1}}|w_{m_{1}k_{1}\ell_{1}}(y,t)|^{2}\,dy\\
		& =\|\chi\|_{L^{\infty}(0,T)}^{2}C_{n,s}|B_{3}\setminus\overline{B_{2}}||B_{1}|\left\|w_{m_{1}k_{1}\ell_{1}}\right\|_{L^{\infty}(0,T;L^{2}(B_{1})}^{2},
	\end{align*}
	for some positive constant $C_{n,s}$ depending only on $n$ and $s\in (0,1)$. Combining the estimate above and \eqref{lem:opt-basic-est} gives
	\begin{align}\label{eq:opt-matrix-repn-est1}
			\sup_{t\in(0,T)}\left|\Gamma_{m_{1}k_{1}\ell_{1}}^{m_{2}k_{2}\ell_{2}}(q)(t)\right| \le C_{R,T,n,s}\|\chi\|^2_{W^{2,\infty}(0,T)}e^{-c_{n,s}(m_{1}+k_{1})}.
	\end{align}
	Since $q$ is real-valued, $\Gamma(q)$ is self-adjoint. Therefore,
	we have 
	\begin{equation}\label{eq:opt-matrix-repn-est2}
		\begin{split}
			\sup_{t\in(0,T)}\left|\Gamma_{m_{1}k_{1}\ell_{1}}^{m_{2}k_{2}\ell_{2}}(q)\right|=&\left|\Gamma_{m_{2}k_{2}\ell_{2}}^{m_{1}k_{1}\ell_{1}}(q)\right|\\
			\le & C_{R,T,n,s}\|\chi\|_{W^{2,\infty}(0,T)}^{2} e^{-c_{n,s}(m_{2}+k_{2})}.
		\end{split}
	\end{equation}
	The estimate \eqref{eq:opt-matrix-repn-est3} then follows from \eqref{eq:opt-matrix-repn-est1} and \eqref{eq:opt-matrix-repn-est2}.\end{proof}

\subsection{Construction of a family of $\delta$-net}

It follows easily from \eqref{eq:opt-matrix-repn-est3} that
\begin{equation}
\sup_{t\in(0,T)}(1+\sigma)^{n+2}\left|\Gamma_{m_{1}k_{1}\ell_{1}}^{m_{2}k_{2}\ell_{2}}(q)(t)\right|\le C_{R,T,n,s}'\|\chi\|_{W^{2,\infty}(0,T)}^{2} (1+\sigma)^{n+2}e^{-c_{n,s}'\sigma},\label{eq:opt-matrix-repn-est4}
\end{equation}
that is, 
\begin{align*}
\sup_{t\in(0,T)}\left\|(\Gamma_{m_{1}k_{1}\ell_{1}}^{m_{2}k_{2}\ell_{2}}(q)(t))\right\|_{X}
	\le C_{R,T,n,s}'\|\chi\|_{W^{2,\infty}(0,T)}^{2} \sup_{\sigma\ge0}(1+\sigma)^{n+2}e^{-c_{n,s}'\sigma}<\infty,
\end{align*}
and thus $\LC \Gamma_{m_{1}k_{1}\ell_{1}}^{m_{2}k_{2}\ell_{2}}(B_{+,R}^{\infty})(t)\RC \subset X$, for all $t\in(0,T)$. Let us define the $\delta$-net as follows.

\begin{defi}[$\delta$-net]
	A set $Y$ of a metric space $(M,\mathsf{d})$ is called a $\delta$-net
	for $Y_{1}\subset M$ if for any $x\in Y_{1}$, there is a $y\in Y$
	such that $\mathsf{d}(x,y)\le\delta$. 
\end{defi}

\begin{lem}
	\label{lem:opt-delta-net} Let $n \ge 2$. Given any $R>1$ and $0<\delta<\|\chi\|_{W^{2,\infty}(0,T)}^{2}$.
	There exists a family $\left\{ Y(t):\,  t\in(0,T)\right\}$
	such that each $Y(t)$ is a $\delta$-net of 
	$$
	\LC (\Gamma_{m_{1}k_{1}\ell_{1}}^{m_{2}k_{2}\ell_{2}}(B_{+,R}^{\infty})(t)),\norm{\cdot}_{X}\RC
	$$
	and satisfies 
	\begin{equation}
		\log|Y(t)|\le K_{n,s}\log^{2n+1}\left(\frac{K_{R,T,n,s}\|\chi\|_{W^{2,\infty}(0,T)}^{2}}{\delta}\right)\label{eq:opt-delta-net-count1}
	\end{equation}
	for some positive constants $K_{n,s}$ and $K_{R,T,n,s}$, where $|Y(t)|$ denotes the cardinality of $Y(t)$. 
\end{lem}

\begin{proof}
	We first note that it suffices to take $C_{R,T,n,s}'\ge1$
	and $c_{n,s}'>0$ described in Lemma~\ref{lem:opt-matrix-repn-est}. 
	
	\vspace{3mm}
	\noindent{\it Step 1. Initialization.}
	\vspace{3mm}
	
	Given any $0<\delta<\|\chi\|_{W^{2,\infty}(0,T)}^{2}$, suggested by \eqref{eq:opt-matrix-repn-est4},
	let $\tilde{\sigma}>0$ be the unique solution to 
	\begin{equation*}
		\LC 1+\tilde{\sigma}\RC^{n+2}e^{-c_{n,s}'\tilde{\sigma}}=\frac{\delta}{C_{R,T,n,s}'\|\chi\|_{W^{2,\infty}(0,T)}^{2}}.
	\end{equation*}
	It is easy to see that 
	\[
	\frac{\delta}{C_{R,T,n,s}'\|\chi\|_{W^{2,\infty}(0,T)}^{2}}=\LC 1+\tilde{\sigma}\RC^{n+2}e^{-\frac{c'_{n,s}}{2}\tilde{\sigma}}e^{-\frac{c'_{n,s}}{2}\tilde{\sigma}}\le C_{n,s}''e^{-\frac{c'_{n,s}}{2}\tilde{\sigma}}
	\]
	with 
	\[
	C_{n,s}'':=\sup_{\sigma>0}\LC 1+\sigma\RC^{n+2}e^{-\frac{c_{n,s}'}{2}\sigma}.
	\]
	Therefore, we have 
	\[
	\tilde{\sigma}\le\frac{2}{c_{n,s}'}\log\LC\frac{C_{n,s}''C_{R,T,n,s}'\|\chi\|_{W^{2,\infty}(0,T)}^{2}}{\delta}\RC.
	\]
Let $\sigma_{*}=\lfloor\tilde\sigma\rfloor$, then
 	\begin{equation}
		1+\sigma_{*}\le1+\tilde{\sigma}\le{c_{n,s}''}\log\LC \frac{C_{n,s}''C_{R,T,n,s}'\|\chi\|_{W^{2,\infty}(0,T)}^{2}}{\delta}\RC\label{eq:opt-sigma-star-est}
	\end{equation}
	for some constant $c_{n,s}''$. Observe that if ${\mathbb Z}\ni\sigma\ge\sigma_\ast+1$, then
	\begin{equation}\label{eq:opt-sigma-star}
	\LC 1+{\sigma}\RC^{n+2}e^{-c_{n,s}'{\sigma}}\le\frac{\delta}{C_{R,T,n,s}'\|\chi\|_{W^{2,\infty}(0,T)}^{2}},
	\end{equation}
	where $\mathbb{Z}$ denotes the set of integers.
	
	\vspace{3mm}
	\noindent{\it Step 2. Construction of sets.}
	\vspace{3mm}
	
	Let $\delta':=\frac{\delta}{(1+\sigma_{*})^{n+2}}$ and define 
	\[
	Y':=\left\{ a\in\delta'\mathbb{Z}:\,  |a|\le C_{R,T,n,s}'\|\chi\|_{W^{2,\infty}(0,T)}^{2}(1+\sigma_{*})^{-(n+2)}\right\},
	\]
	and 
	\[
	Y(t):=\left\{ \LC  b_{m_{1}k_{1}\ell_{1}}^{m_{2}k_{2}\ell_{2}}(t)\RC :\, \begin{array}{l}
			b_{m_{1}k_{1}\ell_{1}}^{m_{2}k_{2}\ell_{2}}(t)\in Y'\text{ if }{\mathbb Z}_+\ni\sigma\le\sigma_{*},\\
			b_{m_{1}k_{1}\ell_{1}}^{m_{2}k_{2}\ell_{2}}(t)=0, \text{ otherwise, }
	\end{array}\right\},
	\]
	where ${\mathbb Z}_+$ is the set of non-negative integers. 
	
	\vspace{3mm}
	\noindent{\it Step 3. Verifying $Y(t)$ is a $\delta$-net.}
	\vspace{3mm}
	
	Our goal is to construct an appropriate $\LC b_{m_{1}k_{1}\ell_{1}}^{m_{2}k_{2}\ell_{2}}(t)\RC \in Y(t)$ that is an approximation of $\LC \Gamma_{m_{1}k_{1}\ell_{1}}^{m_{2}k_{2}\ell_{2}}(q)(t)\RC$. 
	\begin{itemize}
		\item If $\sigma\le\sigma_{\ast}$, we choose $b_{m_{1}k_{1}\ell_{1}}^{m_{2}k_{2}\ell_{2}}(t)\in Y'$
		to be the closest element to $\Gamma_{m_{1}k_{1}\ell_{1}}^{m_{2}k_{2}\ell_{2}}(q)(t)$.
		Then 
		\[
		(1+\sigma)^{n+2}\left|b_{m_{1}k_{1}\ell_{1}}^{m_{2}k_{2}\ell_{2}}(t)-\Gamma_{m_{1}k_{1}\ell_{1}}^{m_{2}k_{2}\ell_{2}}(q)(t)\right|\le(1+\sigma_{*})^{n+2}\delta'=\delta.
		\]
		\item Otherwise, if $\sigma\ge\sigma_{*}+1$, we simply choose $b_{m_{1}k_{1}\ell_{1}}^{m_{2}k_{2}\ell_{2}}(t)=0$.
		Consequently, we obtain from \eqref{eq:opt-matrix-repn-est4} and \eqref{eq:opt-sigma-star} that 
		\begin{align*}
			& (1+\sigma)^{n+2}\left|b_{m_{1}k_{1}\ell_{1}}^{m_{2}k_{2}\ell_{2}}(t)-\Gamma_{m_{1}k_{1}\ell_{1}}^{m_{2}k_{2}\ell_{2}}(q)(t)\right| \\= &(1+\sigma)^{n+2}\left|\Gamma_{m_{1}k_{1}\ell_{1}}^{m_{2}k_{2}\ell_{2}}(q)(t)\right|\\
			\le& C_{R,T,n,s}'\|\chi\|_{W^{2,\infty}(0,T)}^{2}(1+\sigma)^{n+2}e^{-c_{n,s}'\sigma}\\
			\le & \delta.
		\end{align*}
	\end{itemize}
	Combining these two cases, and by the definition of $X$-norm, we have 
	\[
	\sup_{t\in(0,T)}\left\|\LC b_{m_{1}k_{1}\ell_{1}}^{m_{2}k_{2}\ell_{2}}-\Gamma_{m_{1}k_{1}\ell_{1}}^{m_{2}k_{2}\ell_{2}}(q)\RC\right\|_{X}\le\delta.
	\]
	In other words, we have shown that, $Y(t)$ is a $\delta$-net
	of $\LC (\Gamma_{m_{1}k_{1}\ell_{1}}^{m_{2}k_{2}\ell_{2}}(B_{+,R}^{\infty})(t)),\norm{\cdot}_{X}\RC$, for each $t\in(0,T)$. 
	
	\vspace{3mm}
	\noindent{\it Step 4. Calculating the cardinality of $Y(t)$.}
	\vspace{3mm}
	
	Let $N_{\sigma}$ be the number of 6-tuples $(m_{1},k_{1},\ell_{1},m_{2},k_{2},\ell_{2})$
	with $\sigma=\max\{m_{1}+k_{1},m_{2}+k_{2}\}$ as in Lemma \ref{lem:opt-matrix-repn-est}. We now want to estimate $N_{\sigma}$. First of all, we fix any integer $0\le\sigma'\le\sigma$ and compute the number of 6-tuples with $m_{1}+k_{1}=\sigma$
	and $m_{2}+k_{2}=\sigma'$. It is easy to see that there are 
	\[
	\sigma+1\text{ choices of }m_{1}\quad\text{and}\quad\sigma'+1\text{ choices of }m_{2}.
	\]
	Moreover, the number of choices of $\ell_{i}$ is bounded by $\ell_{m_{i}}$,
	and, from Proposition~\ref{prop:opt-special-basis}, we can see that 
	\[
	\ell_{m_{i}}\le\ell_{\sigma}\le2(1+\sigma)^{n-2}\quad\text{ for }i=1,2.
	\]
	Therefore, the number $N_\sigma$ of 6-tuples with $m_{1}+k_{1}=\sigma$ and
	$m_{2}+k_{2}=\sigma'$ is bounded by $4(1+\sigma)^{2n-2}$. Thus,
	the number of 6-tuples with $m_{1}+k_{1}=\sigma$ and $m_{2}+k_{2}\le\sigma$
	is bounded by $4(1+\sigma)^{2n-1}$. Interchanging the role of $(m_{1},k_{1},\ell_{1})$
	with $(m_{2},k_{2},\ell_{2})$, we obtain a similar bound, and hence
	\[
	N_{\sigma}\le8(1+\sigma)^{2n-1}.
	\]
	Therefore, we derive that 
	\[
	N_{*}:=\sum_{\sigma=0}^{\sigma_{*}}N_{\sigma}\le\sum_{\sigma=0}^{\sigma_{*}}8(1+\sigma)^{2n-1}\le8(1+\sigma_{*})^{2n}.
	\]
	
	From \eqref{eq:opt-sigma-star-est}, it follows 
	\[
	N_{*}\le8 \LC {c_{n,s}''}\log\LC \frac{C_{n,s}''C_{R,T,n,s}'\|\chi\|_{W^{2,\infty}(0,T)}^{2}}{\delta}\RC\RC^{2n}.
	\]
	Since $|Y(t)|=|Y'|^{N_{*}}$ and 
	
	\begin{align*}
		|Y'| & =1+2\left\lfloor\frac{C_{R,T,n,s}'\|\chi\|_{W^{2,\infty}(0,T)}^{2} (1+{\sigma}_\ast)^{-(n+2)}}{\delta'}\right\rfloor\\
		& \le1+\frac{2C_{R,T,n,s}'\|\chi\|_{W^{2,\infty}(0,T)}^{2}}{\delta},
	\end{align*}
	we can see that
	\begin{align*}
		 \log|Y(t)|=&N_{*}\log|Y'|\\
		\le & 8\LC {c_{n,s}''}\log\LC \frac{C_{n,s}''C_{R,T,n,s}'\|\chi\|_{W^{2,\infty}(0,T)}^{2}}{\delta}\RC\RC^{2n}\\
		& \qquad\times\log \LC 1+\frac{2C_{R,T,n,s}'\|\chi\|_{W^{2,\infty}(0,T)}^{2}}{\delta}\RC.
	\end{align*}
	Setting 
	\begin{align*}
		K_{R,T,n,s} & :=C_{n,s}''C_{R,T,n,s}'+2C_{R,T,n,s}',\\
		\wt{C}_{n,s} & :={c_{n,s}''}+1,
	\end{align*}
	we then obtain 
	\begin{align*}
		\log|Y(t)| & \le8\LC    \wt{C}_{n,s}\log\LC\frac{K_{R,T,n,s}\|\chi\|_{W^{2,\infty}(0,T)}^{2}}{\delta}\RC\RC^{2n+1}\\
		& =K_{n,s}\log^{2n+1}\LC \frac{K_{R,T,n,s}\|\chi\|_{W^{2,\infty}(0,T)}^{2}}{\delta}\RC
	\end{align*}
	with $K_{n,s}=8(\wt{C}_{n,s})^{2n+1}$. This proves the assertion.
\end{proof}

\begin{rmk}
	Note that 
	\[
	\inf_{0<\delta<\|\chi\|_{W^{2,\infty}(0,T)}^{2}}\log\LC \frac{K_{R,T,n,s}\|\chi\|_{W^{2,\infty}(0,T)}^{2}}{\delta}\RC=\log(K_{R,T,n,s}).
	\]
	Therefore, for each $\alpha>0$ and 
	\begin{equation}
		0<\epsilon<\log^{-\frac{(2n+1)\alpha}{n}}(K_{R,T,n,s})=:c_{R,T,n,s},\label{eq:opt-epsilon-arbitrary}
	\end{equation}
	there exists a unique $0<\delta<\|\chi\|_{W^{2,\infty}(0,T)}^{2}$ such
	that 
	\begin{equation}
		\epsilon^{-\frac{n}{(2n+1)\alpha}}=\log\LC\frac{K_{R,T,n,s}\|\chi\|_{W^{2,\infty}(0,T)}^{2}}{\delta}\RC.\label{eq:opt-delta-epsilon}
	\end{equation}
	Therefore, we can rewrite \eqref{eq:opt-delta-net-count1} as 
	\begin{equation}
		\log|Y(t)|\le K_{n,s}\epsilon^{-\frac{n}{\alpha}}.\label{eq:opt-delta-net-count2}
	\end{equation}
\end{rmk}

\subsection{Construction of an $\epsilon$-discrete set}

Fixing any $r_{0}\in(0,1)$, $\alpha>0$, $\epsilon>0$, and $\beta>0$,
we define the following set: 
\[
\mathcal{N}_{\alpha\beta}^{\epsilon}(B_{r_{0}}):=\left\{ q\ge0:\, {\rm supp}\,(q)\subset B_{r_{0}},\ \|q\|_{L^{\infty}}\le\epsilon,\ \|q\|_{C^{\alpha}}\le\beta \right\}.
\]

The following lemma can be found in \cite[Proposition 2.1]{KUW21instability}, in \cite[Lemma 5.2]{ZZ19instability}, or in \cite{KT59entropy,KT61entropy}
in a more abstract form, also see \cite[Lemma 2]{Man01instability}
for a direct proof, which is valid for all $n \in \mathbb{N}$. Additionally, we refer to \cite[Proposition 3.1]{DR03instability}
and \cite[Proposition 2.2]{DR03instabilityarxiv}, where
similar results were derived under different settings.

\begin{defi}[$\eps$-discrete set]
	A set $Z$ of a metric space $(M,\mathsf{d})$ is called an $\epsilon$-discrete
	set if $\mathsf{d}(z_{1},z_{2})\ge\epsilon$ for all $z_{1}\neq z_{2}\in Z$. 
\end{defi}

\begin{lem}
	\label{lem:opt-discrete-set}Let $n \in \mathbb{N}$ and $\alpha>0$. There exists a constant $\mu=\mu(n,\alpha)>0$
	such that the following statement holds for all $\beta>0$ and for
	all $\epsilon\in(0,\mu\beta)$. Then there exists a $\epsilon$-discrete
	(a.k.a. $\epsilon$-distinguishable) subset $Z$ of $\LC \mathcal{N}_{\alpha\beta}^{\epsilon}(B_{r_{0}}),\|\cdot\|_{L^{\infty}}\RC$
	such that 
	\[
	\log|Z|\ge2^{-(n+1)}\LC\frac{\mu\beta}{\epsilon}\RC^{\frac{n}{\alpha}},
	\]
	where $|Z|$ denotes the cardinality of $Z$.
\end{lem}

\subsection{Proof of Theorem~\ref{thm:main-instability}}

With Lemma~\ref{lem:opt-discrete-set} and Lemma~\ref{lem:opt-discrete-set}
at hand, we can prove the exponential instability of the inverse problem for the fractional
wave equation. 
\begin{proof}
	[Proof of Theorem~{\rm \ref{thm:main-instability}}]Let $\mu$ and $c_{R,T,n,s}$ be
	the constants given in Lemma~\ref{lem:opt-discrete-set} and in \eqref{eq:opt-epsilon-arbitrary}, respectively. For each
	$0<\epsilon<\min \left\{c_{R,T,n,s},R,\mu\beta\right\}$, we can construct an
	$\epsilon$-discrete set $Z$ as described in Lemma~\ref{lem:opt-discrete-set}.
	Let $\delta$ be the constant chosen in \eqref{eq:opt-delta-epsilon}. Next, for
	each $t\in(0,T)$, we construct a $\delta$-net $Y(t)$ as
	in Lemma~\ref{lem:opt-delta-net} and \eqref{eq:opt-delta-net-count2}
	holds. Clearly, $Y(t)$ is also a $\delta$-net for $\LC \Gamma_{m_{1}k_{1}\ell_{1}}^{m_{2}k_{2}\ell_{2}}(Z)(t),\|\cdot\|_{X}\RC$. 
	
	We now choose $\beta=\beta(R,n,\alpha)$ sufficiently large (which
	is independent of $\epsilon$) such that $\mu\beta\ge R$ and 
	\[
	\log|Z|\ge2^{-(n+1)}\LC \frac{\mu\beta}{\epsilon}\RC^{\frac{n}{\alpha}}>K_{n,s}\epsilon^{-\frac{n}{\alpha}}\ge\log|Y(t)|.
	\]
	Therefore, by the pigeonhole principle, for each $t\in(0,T)$, there exist $\LC y_{m_{1}k_{1}\ell_{1}}^{m_{2}k_{2}\ell_{2}}(t)\RC  \in Y(t)$,
	and two different $q_{1},q_{2}\in Z\subset\mathcal{N}_{\alpha\beta}^{\epsilon}(B_{r_{0}})$
	such that 
	\[
	\left\| \LC  \Gamma_{m_{1}k_{1}\ell_{1}}^{m_{2}k_{2}\ell_{2}}(q_{i})(t)-y_{m_{1}k_{1}\ell_{1}}^{m_{2}k_{2}\ell_{2}}(t)\RC\right\|_{X}\le\delta\quad\text{ for }i=1,2.
	\]
	In view of Lemma~\ref{lem:opt-matrix-repn}, we have 
	\begin{align*}
		& \sup_{t\in(0,T)}\left\|\chi(\Lambda_{q_{1}}-\Lambda_{q_{2}})\chi\right\|_{L^{2}(B_{3}\setminus\overline{B_{2}})\rightarrow L^{2}(B_{3}\setminus\overline{B_{2}})}(t)\\
		=&\sup_{t\in(0,T)}\left\|\Gamma(q_{1})-\Gamma(q_{2})\right\|_{L^{2}(B_{3}\setminus\overline{B_{2}})\rightarrow L^{2}(B_{3}\setminus\overline{B_{2}})}(t)\\
		\le&4\left\|\Gamma_{m_{1}k_{1}\ell_{1}}^{m_{2}k_{2}\ell_{2}}(q_{1})(t)-\Gamma_{m_{1}k_{1}\ell_{1}}^{m_{2}k_{2}\ell_{2}}(q_{2})(t)\right\|_{X}\\
		\le & 8\delta=K_{R,T,n,s}\|\chi\|_{W^{2,\infty}(0,T)}^{2} \exp(-\epsilon^{-\frac{n}{(2n+1)\alpha}}).
	\end{align*}
	The arbitrariness of $0\not\equiv\chi\in  C_{c}^{\infty}((0,T))$ leads to the estimate \eqref{eq:opt-DN-map-small}, while the estimate \eqref{eq:opt-discrete}
	immediately follows form the definition of $Z$. Moreover, since $\epsilon<R$, $\|q_{i}\|_{L^{\infty}}\le R$, for $i=1,2$. The proof is now completed. 
\end{proof}


We next prove the exponential instability in the case of $n=1$, Theorem~\ref{thm:1dim-main-instability}. The proof of Theorem~\ref{thm:1dim-main-instability} is very similar
to that of Theorem~{\rm \ref{thm:main-instability}}. The main
difference is that when $n=1$, the boundary $\partial B_{1}$ of the interval
$B_{1}=(-1,1)$ consists only two points $\{-1,1\}$. Therefore, we need to modify the proof of Proposition~\ref{prop:opt-special-basis}. 

We first construct an orthonormal basis $\{Y_{k}\}$ of $L^{2}((2,3))$
such that the solution $\tilde{Y}_{k}$ of 
\begin{equation}
	\begin{cases}
		(-\Delta)^{s}\tilde{Y}_{k}=0 & \text{ in }B_{1},\\
		\tilde{Y}_{k}=\mathbbm1_{(2,3)}Y_{k} & \text{ in }\mathbb{R}^{1}\setminus\overline{B_{1}},
	\end{cases}\label{eq:1-exterior-problem}
\end{equation}
satisfies some exponential decay bound. Similar to the proof of \cite[Lemma 2.1]{RS18Instability},
using the Poisson formula of $u_{k}$ in \cite[Theorem 2.10]{Bucur16Green},
there exists a constant $c=c(s)\neq0$ such that 
\begin{equation}
	\frac{\tilde{Y}_{k}(x)}{c(1-x^{2})^{s}}=\int_{\mathbb{R}^{1}\setminus\overline{B_{1}}}\frac{1}{|x-r|}\frac{\mathbbm1_{(2,3)}(r)Y_{k}(r)}{(r^{2}-1)^{s}}\,dr=\int_{2}^{3}\frac{1}{r-x}\frac{Y_{k}(r)}{(r^{2}-1)^{s}}\,dr\label{eq:1-Poisson-formula1}
\end{equation}
for all $x\in(-1,1)$. If we choose $\left\{Y_{k}=e^{2\pi ik(x-2)}\right\}$ to be the usual orthonormal
basis of $L^{2}((2,3))$, it will be difficult to obtain an exponential decay bound for $\tilde Y_{k}$. Therefore,
we would like to find another orthonormal basis for $L^{2}((2,3))$ to meet our goal. 
\begin{prop}
	\label{prop:1-basis-1dim}There exists a \emph{real-valued} orthonormal
	basis $\left\{Y_{k}\right\}$ of $L^{2}((2,3))$ satisfying that 
	\[
	\|\tilde{Y}_{k}\|_{L^{2}(B_{1})}\le C'e^{-C''k}\quad\text{ for all }\;\,k=0,1,2,\ldots,
	\]
	for some positive constants $C'$ and $C''$ independent of $Y_k$ and $\tilde{Y}_k$, where $\tilde{Y}_{k}\in H^{s}(\mathbb{R}^{1})$ is the unique solution
	of \eqref{eq:1-exterior-problem}. 
\end{prop}

\begin{proof}
	In view of \eqref{eq:1-Poisson-formula1}, we want to find \emph{real-valued}
	$Y_{k}$ of the form 
	\begin{equation}
		Y_{k}(r)=r(r^{2}-1)^{s}g_{k}(r).\label{eq:1-Poisson-anzats}
	\end{equation}
	Plugging the ansatz \eqref{eq:1-Poisson-anzats} into \eqref{eq:1-Poisson-formula1},
	we obtain that for $x\in(-1,1)$
	\begin{align}
		\tilde{Y}_{k}(x) & =c(1-x^{2})^{s}\int_{2}^{3}\frac{1}{1-\frac{x}{r}}g_{k}(r)\,dr\nonumber \\
		& =c(1-x^{2})^{s}\int_{2}^{3}\sum_{j=0}^{\infty}\LC\frac{x}{r}\RC^{j}g_{k}(r)\,dr\nonumber \\
		& =c(1-x^{2})^{s}\sum_{j=0}^{\infty}x^{j}\int_{2}^{3}r^{-j}g_{k}(r)\,dr,\label{eq:1-Poisson-formula2}
	\end{align}
	provided $g_{k}\in L^{1}((2,3))$. 
	\begin{subequations}
		In order to derive the desired decaying properties, for each $k\ge1$,
		we will choose $g_{k}$ such that 
		\begin{equation}
			\int_{2}^{3}r^{-j}g_{k}(r)\,dr=0\quad\text{ for all }\;\,0\le j\le k-1.\label{eq:1-basis-arrangement1}
		\end{equation}
		From \eqref{eq:1-Poisson-anzats}, we also require $g_{k}$ to satisfy
		\begin{equation}
			\delta_{k\ell}=\int_{2}^{3}Y_{k}(r)Y_{\ell}(r)\,dr=\int_{2}^{3}r^{2}(r^{2}-1)^{2s}g_{k}(r)g_{\ell}(r)\,dr.\label{eq:1-basis-orthonormality1}
		\end{equation}
	\end{subequations}
	Setting 
	\[
	h_{k}(r):=r^{2}(r^{2}-1)^{2s}g_{k}(r),
	\]
	we can rewrite \eqref{eq:1-basis-arrangement1} and \eqref{eq:1-basis-orthonormality1}
	as 
	\begin{subequations}
		\begin{align}
			\LC h_{k},r^{-j}\RC_{s}=0, & \quad\text{ for all }0\le j\le k-1,\label{eq:1-basis-arrangement2}\\
			\LC h_{k},h_{\ell}\RC_{s}=\delta_{k\ell}, & \quad\text{ for all non-negative integers }k,\ell,\label{eq:1-basis-orthonormality2}
		\end{align}
	\end{subequations}
	where 
	\[
	\LC  h_{1},h_{2}\RC_{s}:=\int_{2}^{3}r^{-2}(r^{2}-1)^{-2s}h_{1}(r)h_{2}(r)\,dr.
	\]
	Using the Gram-Schmidt process, we can choose 
	\[
	h_{k}(r)\in{\rm span}\,\LC\bigcup_{j=0}^{k}\left\{r^{-j}\right\}\RC\quad\text{ for all }k=0,1,2,\cdots
	\]
	which satisfy \eqref{eq:1-basis-arrangement2} and \eqref{eq:1-basis-orthonormality2}.
	In other words, 
	\[
	\left\{ Y_{k}(r)=r^{-1}(r^{2}-1)^{-s}h_{k}(r): \, k=0,1,2,\ldots \right\}
	\]
	forms an orthonormal basis of $L^{2}((2,3))$. 
	
	We observe that 
	\begin{align*}
		&\quad \left|\int_{2}^{3}r^{-j}g_{k}(r)\,dr\right| \\
		& =\left|\int_{2}^{3}\left(r^{-1}(r^{2}-1)^{-s}h_{k}(r)\right)\left(r^{-1}(r^{2}-1)^{-s}r^{-j}\right)\,dr\right|\\
		& \le \LC \int_{2}^{3}r^{-2}(r^{2}-1)^{-2s}|h_{k}(r)|^{2}\,dr\RC^{\frac{1}{2}}\LC\int_{2}^{3}r^{-2-2j}(r^{2}-1)^{-2s}\,dr\RC^{\frac{1}{2}}\\
		& =\LC\int_{2}^{3}r^{-2-2j}(r^{2}-1)^{-2s}\,dr\RC^{\frac{1}{2}}\\
		& \le2^{-1-j},
	\end{align*}
	for all $k>j$. Combining this estimate with \eqref{eq:1-basis-arrangement1}, we
	have 
	\begin{equation}\label{eq:1-basis-arrangement3}
		\left|\int_{2}^{3}r^{-j}g_{k}(r)\,dr\right|\le\mathbbm1_{\{k>j\}}2^{-1-j}.
	\end{equation}
	Plugging \eqref{eq:1-basis-arrangement3} into \eqref{eq:1-Poisson-formula2},
	we obtain that
	\begin{align*}
		\left|\tilde{Y}_{k}(x)\right| & \le|c|\LC 1-x^{2}\RC^{s}\sum_{j=0}^{\infty}|x|^{j}\left|\int_{2}^{3}r^{-j}g_{k}(r)\,dr\right|\\
		& \le C\sum_{j=k+1}^{\infty}2^{-j}=C\frac{2^{-k-1}}{1-\frac{1}{2}}=C2^{-k},
	\end{align*}
	which is our desired result. 
\end{proof}
Given any bounded linear operator $\mathcal{A}:L^{2}((2,3))\rightarrow L^{2}((2,3))$,
we define 
\[
a_{k_{1}}^{k_{2}}:=\LC \mathcal{A}Y_{k_{1}},Y_{k_{2}}\RC _{L^{2}((2,3))}.
\]
Let $\LC a_{k_{1}}^{k_{2}}\RC$ be the tensor with entries $a_{k_{1}}^{k_{2}}$,
and consider the following Banach space: 
\[
X':=\left\{\LC a_{k_{1}}^{k_{2}}\RC :\,  \left\| \LC a_{k_{1}}^{k_{2}}\RC \right\|_{X'}<\infty\right\},
\]
where 
\[
\left\| \LC  a_{k_{1}}^{k_{2}}\RC \right\|_{X'}:=\sup_{k_{1},k_{2}}\LC 1+\max\{k_{1},k_{2}\}\RC^{3}\left|a_{k_{1}}^{k_{2}}\right|.
\]
Similar to Lemma~\ref{lem:opt-matrix-repn}, we can prove the following lemma. 
\begin{lem}
	\label{lem:1-norm-change}We have 
	\begin{equation}
		\|\mathcal{A}\|_{L^{2}(B_{3}\setminus\overline{B_{2}})\rightarrow L^{2}(B_{3}\setminus\overline{B_{2}})}\le2 \left\| \LC a_{k_{1}}^{k_{2}}\RC \right\|_{X'}.\label{eq:1-Hilbert-schmidt}
	\end{equation}
\end{lem}

\begin{proof}
	In view of the Hilbert-Schmidt norm, we obtain 
	\begin{align}\label{eq:1-Hilbert-schmidt1}
		\begin{split}
			& \|\mathcal{A}\|_{L^{2}(B_{3}\setminus\overline{B_{2}})\rightarrow L^{2}(B_{3}\setminus\overline{B_{2}})}^{2}\le\sum_{k_{1},k_{2}}\left|a_{k_{1}}^{k_{2}}\right|^{2} \\
		    \le & \sup_{k_{1},k_{2}}\LC 1+\max\{k_{1},k_{2}\}\RC^{6}\left|a_{k_{1}}^{k_{2}}\right|^{2}\sum_{k_{1},k_{2}\ge0}\LC 1+\max\{k_{1},k_{2}\}\RC^{-6}.
		\end{split}
	\end{align}
	We also note that 
	\begin{equation}
		\sum_{k_{1},k_{2}\ge0}\LC 1+\max\{k_{1},k_{2}\}\RC^{-6}\le2\sum_{k=0}^{\infty}(1+k)^{-6}\le4.\label{eq:1-Hilbert-schmidt2}
	\end{equation}
	Combining \eqref{eq:1-Hilbert-schmidt1} and \eqref{eq:1-Hilbert-schmidt2} implies \eqref{eq:1-Hilbert-schmidt}. 
\end{proof}

Similar to preceding discussions, let us consider the one spatial dimensional case.

\vspace{2mm}

\noindent $\bullet$ \textbf{Special weak solutions.}

\vspace{2mm}

Let $\chi=\chi(t)\in C_{c}^{\infty}((0,T))$. By the same proof of Lemma~\ref{lem:patch2-veryweak-wellposed}, we can establsih
\begin{lem}
	\label{lem:1-opt-basic-est}If $q\in B_{+,R}^{\infty}$, then there
	exists a unique weak solution $u_{k}$to 
	\[
	\begin{cases}
		\LC \partial_{t}^{2}+(-\Delta)^{s}+q \RC u_{k}=0 & \text{ in }\;\;(B_{1})_{T},\\
		u_{k}(x,t)=\chi(t)\mathbbm1_{(2,3)}Y_{k} & \text{ in }\;\;(\mathbb{R}^{1}\setminus\overline{B_{1}})_{T},\\
		u_{k}=\partial_{t}u_{k}=0 & \text{ in }\;\;\mathbb{R}^{1}\times\{0\},
	\end{cases}
	\]
	such that 
	\[
	\|u_{k}\|_{L^{\infty}(0,T;L^{2}(B_{1}))}\le C_{R,T,s}\|\chi\|_{W^{2,\infty}(0,T)}e^{-c_{s}k}.
	\]
\end{lem}


\vspace{2mm}

\noindent $\bullet$ \textbf{Matrix representation.}

\vspace{2mm}

Again, consider the mapping 
\[
\Gamma(q)(f):=\chi\Lambda_{q}(\chi f)-\chi\Lambda_{0}(\chi f)\quad\text{ for all }f\in  C_{c}^{\infty}((\mathbb{R}^{1}\setminus\overline{B_{1}})_{T}).
\]
In this case, we define 
\[
\Gamma_{k_{1}}^{k_{2}}(q)(t):=\LC \Gamma(q)Y_{k_{1}},Y_{k_{2}}\RC_{L^{2}((2,3))}(t).
\]
Likewise, $\Lambda_{q}$ is self-adjoint, and we have
\begin{equation*}
	\Gamma_{k_{1}}^{k_{2}}(q)(t)=\Gamma_{k_{2}}^{k_{1}}(q)(t).
\end{equation*}
Following the same proof of Lemma~\ref{lem:opt-matrix-repn-est}, we can prove that
\begin{lem}
	\label{lem:1-matrix-repn-basic-est}Given any $q\in B_{+,R}^{\infty}$,
	there exist constants $C_{R,T,s}>1$ and $c_{s}'>0$ such that 
	\[
	\sup_{t\in(0,T)}\left|\Gamma_{k_{1}}^{k_{2}}(q)(t)\right|\le C_{R,T,s}'\|\chi\|_{W^{2,\infty}(0,T)}^{2} e^{-c_{s}'\sigma},
	\]
	where $\sigma:=\max\{k_{1},k_{2}\}$. 
\end{lem}

\vspace{2mm}

\noindent $\bullet$ \textbf{Construction of a family of $\delta$-net.}

\vspace{2mm}

We now construct a $\delta$-net for $(\Gamma_{k_{1}}^{k_{2}}(B_{+,R}^{\infty})(t))$. 
\begin{lem}
	\label{lem:1-delta-net}Given any $R>1$ and $0<\delta<\|\chi\|_{W^{2,\infty}(0,T)}^{2}$.
	There exists a family $\left\{ Y(t): \,  t\in(0,T) \right\}$
	such that each $Y(t)$ is a $\delta$-net of $\LC (\Gamma_{k_{1}}^{k_{2}}(B_{+,R}^{\infty})(t)), \norm{\cdot}_{X'}\RC$
	and satisfies 
	\[
	\log|Y(t)|\le K_{s}\log^{3}\LC\frac{K_{R,T,s}\|\chi\|_{W^{2,\infty}(0,T)}^{2}}{\delta}\RC
	\]
	for some positive constants $K_{s}$ and $K_{R,T,s}$, where $|Y(t)|$ denotes the cardinality of $Y(t)$. 
\end{lem}

\begin{proof}
	The proof of Lemma~\ref{lem:1-delta-net} is almost identically to
	the proof of Lemma~\ref{lem:opt-delta-net} with some minor adjustments
	in Step 4. Here, let $N_{\sigma}$ be the number of 2-tuples $(k_{1},k_{2})$
	with $\max\{k_{1},k_{2}\}=\sigma$. In this case, we can easily obtain
	\[
	N_{\sigma}\le2(1+\sigma)\le8(1+\sigma)^{2n-1}
	\]
	with $n=1$. 
\end{proof}

\begin{proof}[Proof of Theorem~\ref{thm:1dim-main-instability}]
Finally, we can prove Theorem~\ref{thm:1dim-main-instability}
by following the lines in the proof of Theorem~{\rm \ref{thm:main-instability}}.
\end{proof}

\appendix

\section{Proofs related to the forward problem}\label{Appendix}

     In the end of this work, we prove Theorem \ref{thm:well-posedness} in details for the self-containedness.
	The proof of Theorem \ref{thm:well-posedness} is similar to the proof of the case $s=1$, i.e., the well-posedness of the classical wave equation. The main difference is that the estimates and results hold in the fractional Sobolev space. Here we will utilize similar ideas shown in \cite[Chapter 7]{Evan}: The \emph{Galerkin approximation}.

We now set up the Galerkin approximation for the fractional wave equation. To this end, let us consider an eigenbasis $\{w_k\}_{k\in \N}$ associated with the Dirichlet fractional Laplacian in a bounded domain $\Omega$, that is, 
\begin{align*}
	\begin{cases}
		(-\Delta)^s w_k =\lambda_k w_k & \text{ in } \Omega , \\
		w_k =0 & \text{ in }\Omega_e.
	\end{cases}
\end{align*}
Moreover, we can normalize these eigenfunctions so that 
\begin{align}\label{orthogonal H^s}
	\{w_k\}_{k\in \N}\text{ be an orthogonal basis in  }\wt H^s(\Omega),
\end{align} and 
\begin{align}\label{orthonormal L^2}
	\{w_k\}_{k\in \N} \text{ be an orthonormal basis in } L^2(\Omega).
\end{align} Given any fixed integer $m\in \N$, consider the function 
\begin{align}\label{G-approximate solution}
	\bm{v}_m(t):=\sum_{k=1}^m d^k_m w_k,
\end{align}
where the coefficients $d^k_m(t)$ ($0\leq t\leq T$, $k = 1,\ldots, m$) satisfy
\begin{align}\label{G-condition 1}
	\begin{cases}
		d_m^k (0)= (\wt \varphi , w_k), \\
		\LC d_m^k \RC ' (0)=(\wt \psi , w_k),
	\end{cases}
\end{align}
and, for $0\leq t\leq T$,
\begin{align}\label{G-condition 2}
	\LC  \bm{v}_m'' , w_k \RC _{L^2(\Omega)}+B[\bm{v}_m, w_k ;t]=\LC 
	\bm{\wt F}, w_k \RC_{L^2(\Omega)}
	\end{align}
with $k=1,\ldots, m$. Let us split the proof of Theorem \ref{thm:well-posedness} into the following lemmas.

\begin{lem}[Existence of the approximate solution]
	For any $m\in \N$, there exists a unique function $\bm{v}_m$ of the form \eqref{G-approximate solution} satisfying \eqref{G-condition 1} and \eqref{G-condition 2}.
\end{lem}

\begin{proof}
	Due to the orthonormality property \eqref{orthonormal L^2} of $\{w_k\}_{k\in \N}\subset L^2(\Omega)$, we have 
	\begin{align*}
		\LC \bm{v}_m''(t), w_k \RC _{L^2(\Omega)}= \LC d_m^k \RC '' (t).
	\end{align*}
	In addition, we have 
	\begin{align*}
		B[\bm{v}_m, w_k ;t] =\sum _{\ell=1}^m e^{k\ell}d^\ell _m (t),
	\end{align*}
	where $e^{k\ell}:=B[w_\ell,w_k]$ for $k,\ell=1,\ldots, m$. Let us write $F^k(t):=(\bm{\wt F}(t), w_k)$ for $k=1,\ldots,m$. As a result, by using \eqref{G-condition 2} becomes the linear system of ordinary differential equation (ODE)
	\begin{align}\label{G-condition 3}
		\LC d_m^k  \RC ''(t) + \sum_{\ell=1}^m e^{k\ell}d^\ell _m (t)=F^k(t), \text{ for }0\leq t\leq T, \ k=1,\ldots,m,
	\end{align}
	with the initial conditions \eqref{G-condition 1}. Via the standard ODE theory, there exists a unique $C^2$ solution $d_m(t)=(d_m^1(t),\ldots, d^m_m(t))$ satisfying \eqref{G-condition 1}, and solving \eqref{G-condition 3} for $0\leq t\leq T$.
\end{proof}	

Our next goal is to take $m\to \infty$, whenever we have a suitable energy estimate, uniform in $m\in\N$.

\begin{lem}[Energy estimate]
	Under the assumptions of Theorem~{\rm \ref{thm:well-posedness}}, there exists a constant $C>0$, independent of $m\in \N$, such that 
	\begin{align}\label{energy estimate of v_m}
		\begin{split}
		 &	\max_{0\leq t\leq T}\LC  \norm{\bm{v}_m(t)}_{\wt H^s(\Omega)} +\norm{\bm{v}_m'(t)}_{L^2(\Omega)} \RC + \norm{\bm{v}_m''}_{L^2(0,T;H^{-s}(\Omega))} \\
			& \qquad \leq C \LC \norm{{\wt F}}_{L^2 (0,T;L^2(\Omega))} + \norm{\wt\varphi}_{\wt H^s(\Omega)} +\norm{\psi}_{L^2(\Omega)} \RC
		\end{split}
	\end{align}
    for all $m\in \N$.
\end{lem}

\begin{proof}
	We divide the proof into several steps.
	
	\vspace{3mm}
	
	\noindent{\it Step 1. Basic estimates.}
	
	\vspace{3mm}
	
	Multiplying the equation \eqref{G-condition 2} by $\LC d_m^k\RC ' (t)$, and summing over $k=1,\ldots,m$, with the condition \eqref{orthonormal L^2} at hand, we have 
	\begin{align}\label{G-condition 4}
		\LC \bm{v}_m'', \bm{v}_m' \RC_{L^2(\Omega)} + B[\bm{v}_m, \bm{v}_m';t]=\LC \bm{\wt F}, \bm{v}_m'\RC_{L^2(\Omega)}
 	\end{align}
    for a.e. $0\leq t\leq T$. Note that the first term of \eqref{G-condition 4} can be written as 
    \begin{align}\label{G-condition 5}
    	\LC \bm{v}_m'', \bm{v}_m' \RC_{L^2(\Omega)}=\frac{d}{dt}\LC\frac{1}{2}\norm{\bm{v}_m'}^2 _{L^2(\Omega)} \RC.
    \end{align}
    On the other hand, we can express
    \begin{align}\label{G-condition 6}
    	\begin{split}
    		B[\bm{v}_m,\bm{v}_m';t]=&\int_{\R^n} (-\Delta)^{s/2}\bm{v}_m (-\Delta)^{s/2}\bm{v}_m' \, dx +\int_{\Omega} q \bm{v}_m \bm{v}_m' \, dx \\
    		=& \frac{d}{dt}\LC \frac{1}{2} \int_{\R^n} |(-\Delta)^{s/2} \bm{v}_m|^2\, dx \RC+\int_{\Omega} q \bm{v}_m \bm{v}_m'\, dx.
    	\end{split}
    \end{align}
     Meanwhile, we recall that the Hardy-Littlewood-Sobolev inequality
     \begin{align}\label{H-L-S inequality}
     	\norm{\bm{v}_m}_{L^2(\Omega)}\leq C \norm{\bm{v}_m}_{L^\frac{2n}{n-s}(\R^n)}\leq C_{n,s}\norm{(-\Delta)^{s/2} \bm{v}_m}_{L^2(\R^n)},
     \end{align}
     holds for $\bm{v}_m \in \wt H^s(\Omega)$, see e.g. \cite[Proposition 15.5]{ponce2016elliptic}. Indeed, the Hardy-Littlewood-Sobolev inequality also can be further refined in terms of fractional gradient of order $s$ (a.k.a.) $s$-gradient, see e.g. \cite[Section 15.2]{ponce2016elliptic} for more details. 
     Putting together \eqref{G-condition 4}, \eqref{G-condition 5}, \eqref{G-condition 6}, and \eqref{H-L-S inequality}, we can derive the following inequality 
    \begin{align}\label{G-condition 7}
    \begin{split}
    		&\frac{d}{dt} \LC \norm{\bm{v}_m'}_{L^2(\Omega)}^2 +\norm{(-\Delta)^{s/2}\bm{v}_m}_{L^2(\R^n)}^2 \RC \\
    	\leq & C \LC  \norm{\bm{v}_m'}_{L^2(\Omega)}^2 +\norm{(-\Delta)^{s/2}\bm{v}_m}_{L^2(\R^n)}^2 +\norm{\wt F}_{L^2(\Omega)}^2 \RC ,
    \end{split}
    \end{align}
    for some constant $C>0$.
    
    \vspace{3mm}
    
    \noindent{\it Step 2. Gronwall inequality.}
    
    \vspace{3mm}
    
    We next let 
    \begin{align}\label{eta(t)}
    	\eta (t):=\norm{\bm{v}_m'(t)}_{L^2(\Omega)}^2 +\norm{(-\Delta)^{s/2}\bm{v}_m(t)}_{L^2(\R^n)}^2,
    \end{align} 
     and 
     \begin{align}\label{zeta(t)}
     	\zeta(t):=\norm{\wt F(t)}_{L^2(\Omega)}^2,
     \end{align}
     for $0\leq t\leq T$. Then \eqref{G-condition 7} yields that 
     \begin{align*}
      \eta'(t)\leq C_1 \eta (t)+ C_2 \zeta (t), \text{ for }0\leq t\leq T,
     \end{align*}
    for some constants $C_1,C_2>0$. Therefore, the Gronwall's inequality implies that 
    \begin{align}\label{G-condition 8}
    	\eta (t)\leq e^{C_1 t}\LC \eta(0) +C_2 \int_0^t \zeta(s)\, ds \RC, \text{ for }0\leq t \leq T.
    \end{align}
    On the other hand, 
    \begin{align*}
    	\eta(0)=&\norm{\bm{v}_m'(0)}_{L^2(\Omega)}^2 +\norm{(-\Delta)^{s/2}\bm{v}_m(0)}_{L^2(\R^n)}^2 \\
    	\leq & C\LC \norm{\wt \varphi}_{\wt H^s(\Omega)} + \norm{\wt \psi}_{L^2(\Omega)}  \RC,
    \end{align*}
    where we have utilized \eqref{orthogonal H^s}, \eqref{orthonormal L^2} and $\norm{(-\Delta)^{s/2}\bm{v}_m(0)}_{L^2(\R^n)} \leq C\norm{\wt \varphi}_{\wt H^s(\Omega)}$, for some constant $C>0$.
    Thus, combining \eqref{eta(t)}, \eqref{zeta(t)}, and \eqref{G-condition 8}, we derive the following bound 
    \begin{align*}
    	&\norm{\bm{v}_m'(t)}_{L^2(\Omega)}^2 + \norm{(-\Delta)^{s/2}\bm{v}_m}_{L^2(\R^n)}^2 \\ 
    	& \qquad \leq  C\LC \norm{\wt \varphi}_{\wt H^s(\Omega)}^2 + \norm{\wt\psi}_{L^2(\Omega)}^2 +\norm{\wt F}_{L^2(0,T;L^2(\Omega))}^2 \RC.
    \end{align*}
    Since the above estimate is independent of $t\in [0,T]$, one can conclude that 
    \begin{align*}
    	&\max_{0\leq t \leq T}\left(\norm{\bm{v}_m'(t)}_{L^2(\Omega)}^2 + \norm{(-\Delta)^{s/2}\bm{v}_m(t)}_{L^2(\R^n)}^2\right) \\ 
    	& \qquad \leq  C\LC \norm{\wt \varphi}_{\wt H^s(\Omega)}^2 + \norm{\wt\psi}_{L^2(\Omega)}^2 +\norm{\wt F}_{L^2(0,T;L^2(\Omega))}^2 \RC.
    \end{align*}

\vspace{3mm}

\noindent{\it Step 3. Conclusion.}

\vspace{3mm}

For any $\phi \in \wt H^s(\Omega)$ with $\norm{\phi}_{\wt H^s(\Omega)}\leq 1$, we write $\phi=\phi_1+\phi_2$, where $\phi_1 \in \mathrm{span}\{w_k\}_{k=1}^m$ and $(\phi_2,w_k)_{L^2(\Omega)}=0$, for $k=1,\ldots, m$. It is not hard to see $\norm{\phi_1}_{\wt H^s(\Omega)}\leq 1$. In view of \eqref{G-approximate solution} and \eqref{G-condition 2}, we have 
\begin{align*}
	(\bm{v}_m'' , \phi)_{L^2(\Omega)}=(\bm{v}_m '',\phi_1)_{L^2(\Omega)}=(\bm{\wt  F}, \phi_1)- B[\bm{v}_m,\phi_1;t],
\end{align*}
so that 
\begin{align*}
	\left|(\bm{v}_m'', \phi)_{L^2(\Omega)} \right| \leq C \LC \norm{\bm{\wt F}}_{L^2(\Omega)} +\norm{\bm{v}_m}_{\wt H^s(\Omega)}\RC,
\end{align*}
where we used $\norm{\phi_1}_{\wt H^s(\Omega)}\leq 1$. In conclusion, 
\begin{align*}
	\int_0^T \norm{\bm{v}_m ''}_{H^{-s}(\Omega)}^2 \, dt \leq & C \int_0 ^T\LC \norm{\bm{\wt F}}_{L^2(\Omega)}^2 +\norm{\bm{v}_m}_{\wt H^s(\Omega)}^2 \RC dt \\
	\leq & C \LC \norm{\wt \varphi}_{\wt H^s(\Omega)}^2 + \norm{\wt \psi}_{L^2(\Omega)}^2 +\norm{{\wt F}}_{L^2(0,T;L^2(\Omega))}^2\RC.
\end{align*}
This proves the assertion.
\end{proof}

Now, we are ready to prove Theorem \ref{thm:well-posedness}.

\begin{proof}[Proof of Theorem~{\rm \ref{thm:well-posedness}}]
Our goal is to pass the limits in the previous Galerkin approximations.

\vspace{3mm}

\noindent{\it Step 1. Existence of weak solution.}

\vspace{3mm}

Using the energy estimate \eqref{energy estimate of v_m}, it is known that the sequence $\{\bm{v}_m\}_{m\in \N}$, $\{ \bm{v}_m'\}_{m\in \N}$ and $\{\bm{v}_m''\}_{m\in \N}$  are bounded in $L^2(0,T;\wt H^s(\Omega))$, $L^2(0,T;L^2(\Omega))$ and $L^2(0,T;H^{-s}(\Omega))$, respectively.

By extracting a subsequence of $\{\bm{v}_m\}_{m\in \N}$ (still denote the subsequence as $\{\bm{v}_m\}_{m\in \N}$), there exists $\bm{v}\in L^2(0,T;\wt H^s(\Omega))$, with $\bm{v}'\in L^2(0,T;L^2(\Omega))$ and $\bm{v}''\in L^2(0,T;H^{-s}(\Omega))$ such that 
\begin{align*}
	\begin{cases}
		\bm{v}_m \rightharpoonup \bm{v} & \text{ weakly in }L^2(0,T;\wt H^s(\Omega)),\\
		\bm{v}_m' \rightharpoonup  \bm{v}'  & \text{ weakly in }L^2(0,T;L^2(\Omega)),\\
		\bm{v}_m'' \rightharpoonup \bm{v}''  & \text{ weakly in }L^2(0,T;H^{-s}(\Omega)).
	\end{cases}
\end{align*}
Given a fixed integer $N$, choose a function $\bm{\wt v}\in C^1(0,T;\wt H^s(\Omega))$ of the form
\begin{align}\label{function tilde v}
	\bm{\wt v}(t):=\sum_{k=1}^N d^k(t)w_k,
\end{align}
where $\{d^k\}_{k=1}^N$ are smooth functions and $\{w_k\}_{k\in \N}$ are the eigenfunctions given by \eqref{orthogonal H^s} and \eqref{orthonormal L^2}. By choosing $m\geq N$, multiplying \eqref{G-condition 2} by $d^k(t)$, and summing $k=1,\ldots,N$, we then integrate the resulting identity with respect to $t$ to derive 
\begin{align}\label{subseq equation}
	\int_0^T \LC (\bm{v}_m'',\bm{\wt v} )_{L^2(\Omega)} +B[\bm{v}_m,\bm{\wt v};t] \RC dt =\int_0^T (\bm{\wt F}, \bm{\wt v})_{L^2(\Omega)}\, dt.
\end{align}
By passing the limit (along a subsequence, if necessary) in \eqref{subseq equation}, we have 
\begin{align}\label{subseq limit equation}
	\int_0^T \LC (\bm{v}'',\bm{\wt v} )_{L^2(\Omega)} +B[\bm{v},\bm{\wt v};t] \RC dt =\int_0^T (\bm{\wt F}, \bm{\wt v})_{L^2(\Omega)}\, dt.
\end{align}
Note that \eqref{subseq limit equation} holds for all functions $\bm{\wt v}$ of the form \eqref{function tilde v}, which are dense in $L^2(0,T;\wt H^s(\Omega))$. Combining \eqref{subseq limit equation} and the denseness of $\bm{\wt v}$, we obtain 
\begin{align*}
	(\bm{v}'',\phi )_{L^2(\Omega)} +B[\bm{v},\phi ;t] =(\bm{\wt F},\phi )_{L^2(\Omega)},
\end{align*}
for any $\phi \in \wt H^s(\Omega)$ and for a.e. $0\leq t\leq T$. Moreover, by \cite[Theorem~5.9.2]{Evan}, one can show that $\bm{v}\in C([0,T];L^2(\Omega))$ and $\bm{v}'\in C([0,T];H^{-s}(\Omega))$. 

It remains to show that 
\begin{align}\label{check initials}
	\bm{v}(0)=\varphi \quad \text{and}\quad \bm{v}'(0)=\psi.
\end{align}
In order to show \eqref{check initials}, let us select any function $\bm{w}\in C^2([0,T];\wt H^s(\Omega))$, with $\bm{w}(T)=\bm{w}'(T)=0$. Integrating by parts twice with respect to $t$ of \eqref{subseq limit equation} yields that 
\begin{align}\label{uniqueness of initial 1}
	\begin{split}
		&\int_0^T  \LC (\bm{w}'',\bm{v})_{L^2(\Omega)}+B [\bm{v},\bm{w};t] \RC dt\\ =& \int_0^T (\bm{\wt F}, \bm{w})_{L^2(\Omega)}\, dt
		-(\bm{v}(0),\bm{w}'(0))_{L^2(\Omega)}+(\bm{v}'(0),\bm{w}(0))_{L^2(\Omega)}.
	\end{split}
\end{align}
Similarly, from \eqref{subseq equation}, one also has 
\begin{align}\label{uniqueness of initial 2}
	\begin{split}
		&\int_0^T  \LC (\bm{w}'',\bm{v}_m)_{L^2(\Omega)}+B [\bm{v}_m,\bm{w};t] \RC dt\\ =& \int_0^T (\bm{\wt F}, \bm{w})_{L^2(\Omega)}\, dt
		-(\bm{v}_m(0),\bm{w}'(0))_{L^2(\Omega)}+(\bm{v}_m'(0),\bm{w}(0))_{L^2(\Omega)}.
	\end{split}
\end{align}
Hence, by taking $m\to \infty$ of \eqref{uniqueness of initial 2} (along a subsequences as before), we have  
\begin{align}\label{uniqueness of initial 3}
	\begin{split}
		&\int_0^T  \LC (\bm{w}'',\bm{v})_{L^2(\Omega)}+B [\bm{v},\bm{w};t] \RC dt\\ =& \int_0^T (\bm{\wt F}, \bm{w})_{L^2(\Omega)}\, dt
		-(\wt \varphi,\bm{w}'(0))_{L^2(\Omega)}+(\wt \psi,\bm{w}(0))_{L^2(\Omega)}.
	\end{split}
\end{align}
Finally, comparing \eqref{uniqueness of initial 1} and \eqref{uniqueness of initial 3}, by the arbitrariness of $\bm{w}(0)$, $\bm{w}'(0)$, we can conclude that $\bm{v}$ is a weak solution of \eqref{fractional wave zero exterior}. 

\vspace{3mm}

\noindent{\it Step 2. Uniqueness of weak solution.}

\vspace{3mm}

Let $u_{1},u_{2}$ be weak solutions of \eqref{fractional wave well-posedness}. Then $u=u_{1}-u_{2}$ satisfies 
\[
\begin{cases}
	\LC \partial_{t}^{2}+(-\Delta)^{s}+q\RC u=0 & \text{ in }\;\Omega_{T},\\
	u=0 & \text{ in }\;(\Omega_{e})_{T},\\
	u=\partial_{t}u=0 & \text{ in }\;\mathbb{R}^{n}\times\{0\}.
\end{cases}
\]
For each fix $0\le r\le T$, we define 
\[
\bm{w}(t):=\begin{cases}
	\int_{t}^{r}\bm{u}(\tau)\,d\tau & \text{ if }0\le t\le r,\\
	0 & \text{ otherwise.}
\end{cases}
\]
Note that $\bm{w}(t)\in\wt{H}^{s}(\Omega)$ for each $0\le t\le T$. Therefore, choosing the test function $\phi=\bm{w}(t)$ in Definition~\ref{def:weak-soln} yields 
\[
(\bm{u}''(t),\bm{w}(t))_{L^{2}(\Omega)}+B[\bm{u},\bm{w};t]=0\quad\text{for a.e. }0\le t\le T.
\]
Since $\bm{u}'(0)=\bm{w}(r)=0$ and $\bm{w}'(t)=-\bm{u}(t)$, we have 
\begin{align*}
	0 & =\int_{0}^{r}(\bm{u}''(t),\bm{w}(t))_{L^{2}(\Omega)}\,dt+\int_{0}^{r}B[\bm{u},\bm{w};t]\,dt\\
	& =-\int_{0}^{r}(\bm{u}'(t),\bm{w}'(t))_{L^{2}(\Omega)}\,dt+\int_{0}^{r}B[\bm{u},\bm{w};t]\,dt\\
	& =\int_{0}^{r}(\bm{u}'(t),\bm{u}(t))_{L^{2}(\Omega)}\,dt-\int_{0}^{r}B[\bm{w}',\bm{w};t]\,dt\\
	& =\frac{1}{2}\int_{0}^{r}\frac{d}{dt}(\bm{u}(t),\bm{u}(t))_{L^{2}(\Omega)}\,dt-\frac{1}{2}\int_{0}^{r}\frac{d}{dt}B[\bm{w},\bm{w};t]\,dt\\
	& =\frac{1}{2}\|\bm{u}(r)\|_{L^{2}(\Omega)}^{2}-\frac{1}{2}\|\bm{u}(0)\|_{L^{2}(\Omega)}^{2}-\frac{1}{2}B[\bm{w},\bm{w};r]+\frac{1}{2}B[\bm{w},\bm{w};0]\\
	& =\frac{1}{2}\|\bm{u}(r)\|_{L^{2}(\Omega)}^{2}+\frac{1}{2}B[\bm{w},\bm{w};0]\\
	& =\frac{1}{2}\|\bm{u}(r)\|_{L^{2}(\Omega)}^{2}+\frac{1}{2}\|(-\Delta)^{s/2}\bm{w}(0)\|_{L^{2}(\mathbb{R}^{n})}^{2}+\frac{1}{2}\int_{\Omega}q|\bm{w}(0)|^{2}\,dx.
\end{align*}
That is, we obtain 
\begin{equation}
	\|\bm{u}(r)\|_{L^{2}(\Omega)}^{2}+\|(-\Delta)^{s/2}\bm{w}(0)\|_{L^{2}(\mathbb{R}^{n})}^{2}=-\int_{\Omega}q|\bm{w}(0)|^{2}\,dx\le\|q\|_{L^{\infty}}\int_{\Omega}|\bm{w}(0)|^{2}\,dx\label{eq:uniqueness1}
\end{equation}
Since 
\[
\int_{\Omega}|\bm{w}(0)|^{2}\,dx=\int_{\Omega}\left|\int_{0}^{r}\bm{u}(\tau)\,d\tau\right|^{2}\,dx\le\int_{0}^{r}\int_{\Omega}|\bm{u}(t)|^{2}\,dx\,dt=\int_{0}^{r}\|\bm{u}(t)\|_{L^{2}(\Omega)}^{2}\,dt,
\]
\eqref{eq:uniqueness1} implies 
\begin{equation}
	\|\bm{u}(r)\|_{L^{2}(\Omega)}^{2}\le\|q\|_{L^{\infty}}\int_{0}^{r}\|\bm{u}(t)\|_{L^{2}(\Omega)}^{2}\,dt.\label{eq:uniqueness2}
\end{equation}
Multiplying \eqref{eq:uniqueness2} by the integrating factor $e^{-r\|q\|_{L^{\infty}}}$ yields 
\begin{align*}
	& \frac{d}{dr}\left[e^{-r\|q\|_{L^{\infty}}}\int_{0}^{r}\|\bm{u}(t)\|_{L^{2}(\Omega)}^{2}\,dt\right]\\
	 =&e^{-r\|q\|_{L^{\infty}(\Omega)}}\left[-\|q\|_{L^{\infty}}\int_{0}^{r}\|\bm{u}(t)\|_{L^{2}(\Omega)}^{2}\,dt+\|\bm{u}(r)\|_{L^{2}(\Omega)}^{2}\right]\le0,
\end{align*}
that is, 
\[
e^{-r\|q\|_{L^{\infty}(\Omega)}}\int_{0}^{r}\|\bm{u}(t)\|_{L^{2}(\Omega)}^{2}\,dt=0\quad\text{for all }r\in(0,T),
\]
and this immediately implies $\bm{u}\equiv0$. 

\vspace{3mm}

\noindent{\it Step 3. Energy estimate.}

\vspace{3mm}

In Step 1 (precisely, \eqref{energy estimate of v_m}), we have derived 
\begin{align*}
	&\max_{0\le t\le T}\LC \|\bm{v}_{m}(t)\|_{\wt{H}^{s}(\Omega)}+\|\bm{v}_{m}'(t)\|_{L^{2}(\Omega)}\RC\\
	\le& C\LC \|\wt{\bm{F}}\|_{L^{2}(0,T;L^{2}(\Omega))}+\|\wt{\varphi}\|_{\wt{H}^{s}(\Omega)}+\|\psi\|_{L^{2}(\Omega)}\RC. 
\end{align*}
Passing the limit $m\rightarrow\infty$, the estimate \eqref{eq:energy-est0} follows directly from above.
\end{proof}

\vskip0.5cm

\noindent\textbf{Acknowledgments.} 
The second author is partially  supported by the Ministry of Science and Technology Taiwan, under the Columbus Program: MOST-109-2636-M-009-006, 2020-2025. The third author is partly supported by MOST 108-2115-M-002-002-MY3 and 109-2115-M-002-001-MY3.

\bibliographystyle{alpha}
\bibliography{ref}

\end{document}